\documentclass[pdftex]{article}
\usepackage{amsmath,amssymb,amsthm}
\usepackage{geometry} 
\usepackage{indentfirst} 
\usepackage{tikz}
\usepackage{tikz-cd}
\usepackage{hyperref}

\newtheorem{thm}{Theorem}[subsection]
\newtheorem{prop}[thm]{Proposition}
\newtheorem{cor}[thm]{Corollary}
\newtheorem{lem}[thm]{Lemma}

\theoremstyle{definition}
\newtheorem{defi}[thm]{Definition}
\newtheorem{rem}[thm]{Remark}
\newtheorem{eg}[thm]{Example}

\newcommand{\R}{\mathbb{R}}
\newcommand{\Q}{\mathbb{Q}}
\newcommand{\Z}{\mathbb{Z}}
\newcommand{\K}{\mathbb{K}}
\newcommand{\C}{\mathbb{C}}
\newcommand{\Ps}{\mathbb{P}}
\newcommand{\Ox}{\mathcal{O}}
\newcommand{\Fil}{\mathcal{F}}
\newcommand{\Lb}{\mathcal{L}}

\newcommand{\vol}{\mathrm{vol}}
\newcommand{\avol}{\widehat{\vol}}
\newcommand{\adeg}{\widehat{\deg}}
\newcommand{\asy}{\mathrm{asy}}
\newcommand{\an}{\mathrm{an}}
\newcommand{\ndot}{\raisebox{.4ex}{.}}
\newcommand{\ab}{|\ndot|}
\newcommand{\nm}{\|\ndot\|}
\newcommand{\lm}{\lambda_{\max}}

\newcommand{\lmasy}{\lm^\asy}
\newcommand{\Spec}{\mathrm{Spec}\,}
\newcommand{\ord}{\mathrm{ord}}
\newcommand{\dv}{\mathrm{div}}
\newcommand{\red}{\mathrm{red}}
\newcommand{\Div}{\mathrm{Div}}
\newcommand{\WDiv}{\mathrm{WDiv}}
\newcommand{\Pic}{\mathrm{Pic}}

\makeatletter
\@addtoreset{equation}{section}

\makeatother 

\begin{document}
\title{Volume function over a trivially valued field}
\author{Tomoya Ohnishi\footnote{tohnishi@math.kyoto-u.ac.jp}}
\maketitle

\begin{abstract}
We introduce an adelic Cartier divisor over a trivially valued field and discuss the bigness of it.
For bigness, we give the integral representation of the arithmetic volume and prove the existence of limit of it. 
Moreover, we show that the arithmetic volume is continuous and log concave.
\end{abstract}

\tableofcontents

\section{Introduction}
Arakelov geometry is a kind of arithmetic geometry.
Beyond scheme theory, it has been developed to treat a system of equations with integer coefficients such by adding infinite points.
In some sense, it is an extension of Diophantine geometry.
It is started by Arakelov \cite{arakelov1974intersection}, who tried to define the intersection theory on an arithmetic surface.
His result was not complete, but it was done by Faltings \cite{faltings1984calculus}.
He gave the complete intersection theory on a arithmetic surface, such as the arithmetic Riemann-Roch theorem, the Noether formula and so on.
Before a higher dimensional case, we would see the 1-dimensional case, that is, the arithmetic curves.

An arithmetic curve is the spectrum of the integer ring $O_K$ of a number field $K$.
For example, $\Spec \Z$ is an arithmetic curve.
It is not a proper scheme, so it is difficult to treat $\Spec \Z$ like an algebraic curve.
For instance, the degree of principal divisors might not be zero, and the set of Cartier divisors modulo principal divisors is trivial.
This problem is solved by considering the valuation theory.
The scheme $\Spec \Z$ is the set of prime numbers in $\Q$, which corresponds to the finite places of $\Q$.
But the field $\Q$ has not only finite places but also the infinite place.
Hence by adding the point which corresponds to the infinite place to $\Spec \Z$, we can make $\Spec \Z$ ``compact''.
This idea is very successful.
Now we can define an arithmetic divisor on $\Spec \Z$ as a pair of a Cartier divisor and a real number, which is a counterpart of the infinite place (more precisely, it is a continuous $\R$-valued function on a single point set).
The degree of principal divisors is not zero in general, but the arithmetic degree of arithmetic principal divisors is always zero because of the product formula:
\[ |f|_\infty \cdot \prod_{p : \text{primes}} |f|_p = 1, \quad \text{for}\ \forall f \in \Q^\times, \]
where $\ab_\infty$ is the usual absolute value and $\ab_p$ is the $p$-adic absolute value on $\Q$.
Moreover, the set of arithmetic divisors modulo arithmetic principal divisors on $\Spec \Z$ is isomorphic to $\R$.
This result corresponds to the fact that the Picard group of $\Ps^1$ is isomorphic to $\Z$.
In this way like an algebraic curve, we can study the geometry on $\Spec \Z$ equipped with the infinite point.
Above all, the product formula plays an important role.

Next, we briefly recall the higher dimensional case.
An arithmetic variety $X$ is an integral scheme flat and quasi-projective over $\Spec \Z$.
As we equip $\Spec \Z$ with the infinite point, we consider not only $X$ but also the counterpart of the infinite point, which is the analytic space $X(\C)$.
The main tools of Arakelov geometry are the intersection theory and an arithmetic divisor.
The intersection theory on $X$ was given by Gillet-Soul{\'e} \cite{gillet1990arithmetic} by using Green currents on $X(\C)$.
They also proved the general arithmetic Riemann-Roch theorem on $X$ \cite{gillet1992arithmetic}.
One of the most famous application of the arithmetic Riemann-Roch theorem is the solution of the Mordell conjecture due to Voijta \cite{vojta1991siegel}.
As related to arithmetic divisors, it was used to show the existence of small sections of $H^0(X,\overline{D})$, where $\overline{D}$ is an arithmetic divisor.
A small section is a section whose norm is less than or equal 1.
In Arakelov geometry, it plays the same role as a global section in algebraic geometry.
Therefore it is also important to study the asymptotic behavior of the amount of small sections of $H^0(X,n\overline{D})$ as $n \rightarrow \infty$.
To study it, Moriwaki \cite{moriwaki2009continuity} introduced the arithmetic volume
\[ \avol(\overline{D}) = \limsup_{m \rightarrow +\infty} \frac{\hat{h}^0(X,m\overline{D})}{m^d/d!}, \]
where $d = \dim X$ and $\hat{h}^0(X,\overline{D}) = \log \# \{ \text{small sections of}\ H^0(X,\overline{D})\}$.
Moriwaki \cite{moriwaki2009continuity} proved the continuity of the arithmetic volume, Chen and Boucksom \cite{boucksom2011okounkov} proved the concavity.
For general theory, we refer to \cite{moriwaki2014arakelov}.

Recently, by considering the analytic space associated with not only infinite places but also finite places, the study of the adelic version of Arakelov geometry \cite{moriwaki2016adelic} has also developed.
In this theory, the analytic space associated with a non-Archimedean absolute value in the sense of Berkovich plays the main role.
Moreover, as a further generalization, Chen and Moriwaki \cite{chen2019arakelov} introduced the notion of adelic curves.
It is a field equipped with a measure space which is consist of absolute values, and the ``product formula''.
This notion contains several classical settings.
For example, we can treat algebraic curves and arithmetic curves as adelic curves.
However, since the notion of adelic curves is very general, things that cannot be considered as a curve may be an adelic curve.
One of them is a trivially valued field.
It is a field $K$ equipped with the trivial product formula, that is,
\[ |f|_0 = 1, \quad \text{for}\ \forall f \in K^\times. \]
Hence Arakelov geometry over a trivially valued field is the geometry of schemes over an adelic curve $\Spec K$.

As the classical Arakelov geometry, there is the arithmetic volume function of an adelic Cartier divisor $\overline{D}$ on a projective variety $X$, which was introduced by Moriwaki and Chen \cite{chen2017sufficient}:
\[ \avol(\overline{D}) = \limsup_{n \rightarrow \infty} \frac{\adeg_+(n\overline{D})}{n^{d+1}/(d+1)!}, \]
where $d = \dim X$.
In the classical setting, the invariants $\hat{h}^0(X,\overline{D})$ and $\adeg_+(\overline{D})$ behave in a similar way.
For detail, we refer to \cite{chen2010arithmetic}, \cite{chen2015majorations} and \cite{chen2019arakelov}.
We say that an adelic Cartier divisor $\overline{D}$ is big if $\avol(\overline{D}) > 0$.
In this paper, we will show several properties of the arithmetic volume $\avol(\ndot)$.

\begin{thm}
Let $\overline{D},\overline{E}$ be adelic Cartier divisors.
The arithmetic volume has the following properties:
\begin{enumerate}
\item (integral formula).
\[ \avol(\overline{D}) = (d+1)\int_0^\infty F_{\overline{D}}(t)\ \mathrm{d}t, \]
where $F_{\overline{D}}$ is a function given by $\overline{D}$.
(c.f. Theorem \ref{vol1}).
\item (limit existence).
\[ \avol(\overline{D}) = \lim_{n \rightarrow \infty} \frac{\adeg_+(n\overline{D})}{n^{d+1}/(d+1)!}. \]
(c.f. Theorem \ref{vol1}).
\item (continuity). If $D$ is big, we have
\[ \lim_{\epsilon \rightarrow 0} \avol(\overline{D} + \epsilon \overline{E}) = \avol(\overline{D}). \]
(c.f. Theorem \ref{cvol}).
\item (homogeneity). For $a \in \R_{>0}$,
\[ \avol(a\overline{D}) = a^{d+1}\avol(\overline{D}). \]
(c.f. Corollary \ref{vol2}).
\item (log concavity). If $\overline{D},\overline{E}$ are big. we have
\[ \avol(\overline{D} + \overline{E})^{\frac{1}{d+1}} \geq \avol(\overline{D})^{\frac{1}{d+1}} + \avol(\overline{E})^{\frac{1}{d+1}}. \]
(c.f. Theorem \ref{cvol2}).
\end{enumerate}
\end{thm}

Section 2 is devoted to preliminary such as algebraic geometry, normed vector spaces and the Berkovich spaces.
In Section 3, we see the fundamental result of Arakelov geometry over a trivially absolute values.
In Section 4, we discuss the properties of the arithmetic volume.
For example, we will prove the continuity and the concavity of the arithmetic volume.

\section{Preliminary}
\subsection{$\Q$- and $\R$-divisors}
Let $X$ be a variety over a field $K$ and $K(X)$ be a function field of $X$.
By abuse of notation, we also denote the (constant) sheaf of rational functions on $X$ by $K(X)$.
Firstly, we recall the definitions of Cartier divisors and Weil divisors (for detail, see \cite{hartshorne2013algebraic} and \cite{liu2006algebraic}).

\begin{defi}
Let $\Div(X) := H^0(X,K(X)^\times/\Ox_X^\times)$, whose element is called a \textit{Cartier divisor}.
A non-zero rational function $f \in K(X)^\times$ naturally gives rise to a Cartier divisor, which is called a \textit{principal Cartier divisor} (or simply a \textit{principal divisor}) and denoted by $(f)$.
We denote the group law on $\Div(X)$ additive way.
We say that two Cartier divisors $D_1,D_2 \in \Div(X)$ are \textit{linearly equivalent} if $D_1-D_2$ is principal, which is denoted by $D_1 \sim D_2$.
We set $\Pic(X) := \Div(X)/\sim$, which is called the \textit{Picard group of $X$}. 
We say that a Cartier divisor $D \in \Div(X)$ is \textit{effective} if it is contained in the image of the canonical map $H^0(X,\Ox_X \cap K(X)^\times) \rightarrow H^0(X,K(X)^\times/\Ox_X^\times)$.
For two Cartier divisors $D_1,D_2$, we write $D_1 \geq D_2$ if $D_1-D_2$ is effective.
In particular, we write $D \geq 0$ if $D$ is effective.
For an open subset $U$ of $X$, let $D|_U$ be the image of $D$ by the canonical restriction $H^0(X,K(X)^\times/\Ox_X) \rightarrow H^0(U,K(X)^\times/\Ox_X^\times)$, which gives a Cartier divisor on $U$.
\end{defi}

By definition, for $D \in \Div(X)$, there is an open covering $\{U_i\}$ of $X$ such that $D$ is given by some non-zero rational function $f_i \in K(X)^\times$ on $U_i$ and $f_i/f_j \in \Ox_X(U_i \cap U_j)^\times$ for $i \neq j$.
In the above setting, $D$ is effective if and only if $f_i$ is regular on $U_i$, that is, $f_i \in \Ox_X(U_i)$ for all $i$.
 
We can associate any Cartier divisor $D = \{(U_i,f_i)\} \in \Div(X)$ with a subsheaf $\Ox_X(D) \subset K(X)$, which is given by $\Ox_X(D)|_{U_i} := f_i^{-1}\Ox_X|_{U_i}$.
It is well-known that this construction is independent of the choice of a representation $\{(U_i,f_i)\}$ of $D$ and $\Ox_X(D)$ is an  invertible $\Ox_X$-module on $X$.

\begin{prop}[{c.f. \cite[Proposition 6.13]{hartshorne2013algebraic} and \cite[Proposition 1.18]{liu2006algebraic}}]
Let $D_1,D_2$ be Cartier divisors.
\begin{enumerate}
\item $\Ox_X(D_1) \simeq \Ox_X(D_2)$ if $D_1 \sim D_2$.
\item $\Ox_X(D_1 + D_2) \simeq \Ox_X(D_1) \otimes_{\Ox_X} \Ox_X(D_2)$.
\end{enumerate}
\end{prop}

We denote $\Gamma(U,\Ox_X(D))$ by $\Gamma(U,D)$ for an open subset $U$ of $X$.
For any open subset $U$ of $X$, we have
\begin{equation}
\label{sec1}
\Gamma(U,D) = \{f \in K(X)^\times \,|\, (D + (f))|_U \geq 0 \} \cup \{0\}
\end{equation}
by definition.

Conversely, we can associate any invertible $\Ox_X$-module $\Lb$ with a Cartier divisor $D$ such that $\Lb \simeq \Ox_X(D)$.
Let $s$ be a non-zero rational section of $\Lb$, that is, $s \in \Lb_\eta \setminus \{0\}$ where $\eta$ is the generic point of $X$.
Let $\{U_i\}$ be an open covering of $X$ which trivializes $\Lb$, and $\omega_i \in \Lb(U_i)$ be a local basis of $\Lb$ for each $i$.
Then $s$ is denoted by $f_i\omega_i$ on $U_i$ for some $f_i \in K(X)$.
The date $\{(U_i, f_i)\}$ gives the required Cartier divisor $\mathrm{div}(s)$.
For example, if we choose $1$ as a rational section of $\Ox_X(D)$, then we have $\mathrm{div}(1) = D$ by its construction.

\vspace{0.1in}

Next, we assume that $X$ is normal.
Let $X^{(1)} = \{x \in X \,|\, \mathrm{codim}_X\overline{\{x\}} = 1\}$.
For $x \in X^{(1)}$, let $[x] := \overline{\{x\}}$, which is an irreducible closed subset of $X$ and called a \textit{prime divisor}.

\begin{defi}
Let $\WDiv(X) := \bigoplus_{x \in X^{(1)}} \Z [x]$, whose element is called a \textit{Weil divisor}.
If we write 
\[ D=\sum_{x \in X^{(1)}}n_x[x], \]
we set $\ord_x(D) := n_x$.
We say that a Weil divisor $D \in \WDiv(X)$ is \textit{effective} if $\ord_x(D) \geq 0$ for all $x \in X^{(1)}$.
For two Weil divisors $D_1,D_2$, we write $D_1 \geq D_2$ if $D_1-D_2$ is effective.
In particular, we write $D \geq 0$ if $D$ is effective.
For a non-empty open subset $U$ of $X$, let 
\[ D|_U := \sum_{x \in X^{(1)} \cap U} \ord_x(D)[x], \]
which is called the restriction of a Weil divisor $D$ on $U$.
\end{defi}

Let $x \in X^{(1)}.$
Since $X$ is normal, $\Ox_{X,x}$ is a discrete valuation ring.
Hence we have the normalized discrete valuation $\ord_x$ on $K(X)$ associated with $\Ox_{X,x}$.
For a non-zero rational function $f \in K(X)^\times$, let 
\[ (f) := \sum_{x \in X^{(1)}} \ord_x(f) [x]. \]
This is a Weil divisor and such a divisor is called a \textit{principal Weil divisor} (or simply a \textit{principal divisor}).
We say that two Weil divisors $D_1,D_2 \in \WDiv(X)$ are \textit{linearly equivalent} if $D_1-D_2$ is principal.
Then we write $D_1 \sim D_2$.

\vspace{0.1in}

We can associate any Cartier divisor $D \in \Div(X)$ with a Weil divisor as follows:
For any $x \in X^{(1)}$, let $f \in K(X)$ be a local equation around $x$ of $D$.
Then we set $\ord_x(D) := \ord_x(f)$.
It is independent of the choice of a local equation.
Hence we can define that
\[ D := \sum_{x \in X^{(1)}} \ord_x(D)[x]. \]
This construction gives a homomorphism $\varphi:\Div(X) \rightarrow \WDiv(X)$.

\begin{prop}[{c.f. \cite[Proposition 2.14]{liu2006algebraic}}]
\mbox{}
\begin{enumerate}
\item The homomorphism $\varphi$ is injective.
Moreover, $\varphi$ is an isomorphism if $X$ is regular. 
Hence we sometimes identify a Cartier divisor with a Weil divisor.
\item For any $D_1, D_2 \in \Div(X)$, $D_1 \sim D_2$ as Cartier divisors if and only if $D_1 \sim D_2$ as Weil divisors.
\item For any $D \in \Div(X)$, $D \geq 0$ as Cartier divisors if and only if  $D \geq 0$ as Weil divisors.
\end{enumerate}
\end{prop}

We can associate any Weil divisor $D$ with a subsheaf $\Ox_X(D) \subset K(X)$, which is defined by 
\[ \Ox_X(D)|_U := \{f \in K(X)^\times \,|\, (D + (f))|_U \geq 0 \} \cup \{0\} \]
for any open subset $U$ of $X$.
By (\ref{sec1}), if $D$ is Cartier, the above construction gives the same invertible $\Ox_X$-module $\Ox_X(D)$.
However $\Ox_X(D)$ is not invertible if $D$ is not Cartier.

\vspace{0.1in}

Let $\K = \Q$ or $\R$.
Let us introduce the definition of $\K$-divisors.

\begin{defi}
Let $\Div(X)_\K := \Div(X) \otimes_\Z \K$, $\WDiv(X)_\K := \WDiv(X) \otimes_\Z \K$ and $K(X)_\K^\times := K(X)^\times \otimes_\Z \K$.
An element of $\Div(X)_\K$ (resp. $\WDiv(X)_\K$, $K(X)_\K^\times$) is called a \textit{$\K$-Cartier divisor} (resp. a \textit{$\K$-Weil divisor}, a \textit{$\K$-rational function}) on $X$.
Clearly, Cartier divisors and $\Q$-Cartier divisors (resp. Weil divisors and $\Q$-Weil divisors) are $\R$-Cartier divisors (resp. $\R$-Weil divisors).
A non-zero $\K$-rational function $f \in K(X)_\K^\times$ naturally gives rise to a $\K$-Cartier divisor (or equivalently a $\K$-Weil divisor), which is called a \textit{$\K$-principal divisor} and denoted by $(f)$.
We say that two $\R$-Cartier divisors (resp. $\R$-Weil divisors) $D_1,D_2$ are \textit{$\K$-linearly equivalent} if $D_1-D_2$ is $\K$-principal, which is denoted by $D_1 \sim_\K D_2$.
We say that a $\K$-Cartier divisor (resp. a $\K$-Weil divisor) $D$ is \textit{effective} if $D$ is a linear combination of effective divisors with positive coefficients in $\K$.
We write $D_1 \geq D_2$ if $D_1-D_2$ is effective.
In particular, we write $D \geq 0$ if $D$ is effective.
\end{defi}

Similarly to Cartier divisors, for $D \in \Div(X)_\K$, there is an open covering $\{U_i\}$ of $X$ such that $D$ is given by some non-zero $\K$-rational function $f_i \in K(X)_\K^\times$ on $U_i$ and $f_i/f_j \in (\Ox_X(U_i \cap U_j) \otimes_\Z \K)^\times$ for $i \neq j$.

Let $D \in \WDiv(X)_\K$.
By definition, we can write $D = \sum_{x \in X^{(1)}} k_x[x]$, where $k_x \in \K$ and $k_x = 0$ for all but finitely many $x \in X^{(1)}$.
Then we define the round down of $D$ as follows:
\[ \lfloor D \rfloor := \sum_{x \in X^{(1)}} \lfloor k_x \rfloor [x]. \]
This is a Weil divisor and $\lfloor D \rfloor = D$ if and only if $D \in \WDiv(X)$.

For $D \in \WDiv(X)_\K$, the \textit{associated $\Ox_X$-module} $\Ox_X(D)$ is defined by $\Ox_X(\lfloor D \rfloor)$.
Then we have $H^0(X,D) = \{f \in K(X)^\times \,|\, D + (f) \geq 0 \} \cup \{0\}$.
We remark that $D + (f) \geq 0 \Leftrightarrow \lfloor D \rfloor + (f) \geq 0$ for any $f \in K(X)^\times$ and $\Ox_X(2D)$ is not isomorphic to $\Ox_X(D) \otimes_{\Ox_X} \Ox_X(D)$ in general.

\begin{prop}[{c.f. \cite[Theorem 3.2]{liu2006algebraic}}]
Let $D \in \WDiv(X)_\K$.
Then $H^0(X,D)$ is a finite-dimensional vector space over $K$.
\end{prop}

\subsection{Big divisors}
We simply recall the definitions of bigness of Cartier divisors.
Let $X$ be a variety over a field $K$.

\begin{defi}
\label{big}
Let $D$ be a Cartier divisor on $X$.
Let $h^0(D) := \dim_K H^0(X,D)$ and $d = \dim X$.
We define the \textit{volume} $\vol(D)$ of $D$ as follows:
\[ \vol(D) := \limsup_{n \rightarrow +\infty} \frac{h^0(nD)}{n^d/d!}. \]
We say that $D$ is \textit{big} if $\vol(D) > 0$.
\end{defi}

Later we will consider the volume of an $\R$-Weil divisor.
Hence we extends the above definition.

\begin{defi}
Let $D$ be an $\R$-Weil divisor on a normal variety $X$.
We define a function $\mathfrak{h}_D:\R_+ \rightarrow \Z$ by $\mathfrak{h}_D(t) := \dim_K H^0(tD) = \dim_K H^0(\lfloor tD \rfloor)$.
The \textit{volume} of $D$ is defined by
\[ \vol(D) := \limsup_{t \rightarrow +\infty} \frac{\mathfrak{h}_D(t)}{t^d/d!}, \]
where $d = \dim X$. 
We say that $D$ is \textit{big} if $\vol(D) > 0$.
\end{defi}

By Fulger, Koll{\'a}r and Lehmann \cite{fulger2016volume}, the above definition agrees with one in Definition \ref{big} if $X$ is proper and $D$ is Cartier.

Finally we recall the well-known properties of the volume function $\vol(\ndot)$ without a proof (for detail, see \cite{lazarsfeld2017positivity}).

\begin{prop}
Let $X$ be a proper normal variety and $d = \dim X$.
Let $D,E$ be $\R$-Cartier (or $\R$-Weil) divisors on $X$.
\begin{enumerate}
\item $\displaystyle \vol(D) = \lim_{n \rightarrow +\infty} \frac{h^0(nD)}{n^d/d!} \left(= \lim_{t \rightarrow +\infty} \frac{\mathfrak{h}_D(t)}{t^d/d!} \right)$.
\item For $a \in \R_{>0}$, $\vol(aD) = a^d\vol(D)$.
\item The volume function $\vol(\ndot)$ is continuous, that is, $\vol(E) \rightarrow \vol(D)$ as $E \rightarrow D$ (which means that each coefficients of $E$ converge coefficients of $D$ as an $\R$-Weil divisor).
\item The volume function $\vol(\ndot)$ is $d$-concave on big divisors, that is, if $D,E$ are big, then
\[ \vol(D+E)^{1/d} \geq \vol(D)^{1/d} + \vol(E)^{1/d}. \]
\end{enumerate}
\end{prop}

\subsection{Normed vector space}
In this section, we study fundamental properties of a normed vector space over a field equipped with an absolute value.
But  we mainly consider a trivially valued field.

Let $K$ be a field.

\begin{defi}
We say that a map $\ab:K \rightarrow \R_+$ is an \textit{absolute value on $K$} if it satisfies the following conditions:
\begin{enumerate}
\item $\forall a \in K$, $|a|=0 \Leftrightarrow a=0$.
\item $\forall a,b \in K$, $|a|\ndot|b|=|ab|$.
\item (\textit{triangle inequality}) $\forall a,b \in K$, $|a+b| \leq |a|+|b|$.
\end{enumerate}
If an absolute value $\ab$ also satisfies the following inequality
\[ \forall a,b \in K,\ |a+b| \leq \max\{|a|,|b|\}, \]
we say that $\ab$ is \textit{non-Archimedean}.
Otherwise, $\ab$ is called \textit{Archimedean}.
\end{defi}

\begin{defi}
We say that an absolute value $\ab$ on $K$ is \textit{trivial} if it satisfies that $|a|=1$ for any $a \in K \setminus \{0\}$.
A field $K$ equipped with the trivial absolute value $\ab$ is called a \textit{trivially valued field}.
Clearly, the trivial absolute value is non-Archimedean and a trivially valued field is complete.
\end{defi}

Let $V$ be a vector space over $K$.

\begin{defi}
We say that a map $\nm:V \rightarrow \R_+$ is a \textit{(multiplicative) norm over $(K,\ab)$} if it satisfies the following conditions:
\begin{enumerate}
\item $\forall v \in V$, $\|v\|=0 \Leftrightarrow v=0$.
\item $\forall a \in K$ and $v \in V$, $\|av\|=|a|\ndot\|v\|$.
\item (\textit{triangle inequality}) $\forall v,w \in V$, $\|v+w\| \leq \|v\|+\|w\|$.
\end{enumerate}
If a norm $\nm$ also satisfies the following inequality
\[ \forall v,w \in V,\ \|v+w\| \leq \max\{\|v\|,\|w\|\}, \]
we say that $\nm$ is \textit{ultrametric}.
A pair $(V,\nm)$ is called a \textit{normed vector space}.
\end{defi}

Let $V_\bullet = \bigoplus_{n=0}^{\infty} V_n$ be a graded ring over $K$ such that $V_n$ is a vector space over $K$ for all $n$ and $V_0 = K$.
Let $\ab$ be an absolute value on $K$ and $\nm_n$ be a norm of $V_n$ over $(K,\ab)$ for $n \in \Z_{\geq 0}$ such that $\nm_0=\ab$ on $V_0=K$.

\begin{defi}
We say that
\[ (V_\bullet,\nm_\bullet) := \bigoplus_{n=0}^\infty (V_n,\nm_n) \]
is a \textit{normed graded ring over $(K,\ab)$} if $\|v_m \ndot v_n\|_{m+n} \leq \|v_m\|_m\ndot\|v_n\|_n$ for all $v_m \in V_m$ and $v_n \in V_n$. 
\end{defi}

Let $W_\bullet = \bigoplus_{n=0}^\infty W_n$ be a $V_\bullet$-module such that $W_n$ is a vector space over $K$ for all $n$.
Let $h \in \Z_{>0}$.
We say that $W_\bullet$ is a \textit{$h$-graded $V_\bullet$-module} if $v_m\ndot w_n \in W_{hm+n}$ for all $v_m \in V_m$ and $w_n \in W_n$.
If $h=1$, $W_\bullet$ is simply called a \textit{graded $V_\bullet$-module}.

Let $\nm_{W_n}$ be a norm on $W_n$ over $(K,\ab)$ for $n \in \Z_{\geq0}$.

\begin{defi}
We say that
\[ (W_\bullet,\nm_{W_\bullet}) := \bigoplus_{n=0}^\infty (W_n,\nm_{W_n}) \]
is a \textit{normed $h$-graded $(V_\bullet,\nm_\bullet)$-module} if $\|v_m\ndot w_n\|_{W_{hm+n}} \leq \|v_m\|_m\ndot\|w_n\|_{W_n}$ for all $v_m \in V_m$ and $w_n \in W_n$.
If $h=1$, $(W_\bullet,\nm_{W_\bullet})$ is simply called a \textit{normed graded $(V_\bullet,\nm_\bullet)$-module}.
\end{defi}

In the following, let $(V,\nm)$ be an ultrametrically normed vector space over a trivially valued field $(K,\ab)$ and $\dim_K(V) < +\infty$.

\begin{lem}[{c.f. \cite[Proposition 1.1.5]{chen2019arakelov}}]
\label{norm1}
\mbox{}
\begin{enumerate}
\item Let $v_1,\dots, v_n \in V$.
If $\|v_1\|,\dots, \|v_n\|$ are all distinct, then we have $\|v_1+\cdots+v_n\|=\max\{\|v_1\|,\dots,\|v_n\|\}$.
\item $\#\{\|v\| \,|\, v \in V \} \leq \dim_K(V) + 1$.
\end{enumerate}
\end{lem}

\begin{proof}
(1) By induction of $n$, it is sufficient to show in the case of $n=2$.
Let $\|x_1\| > \|x_2\|$.
By definition, we have $\|x_1+x_2\| \leq \|x_1\|$.
On the other hand, $\|x_1\| = \|(x_1 + x_2) - x_2\| \leq \max\{\|x_1+x_2\|,\|x_2\|\}$.
Since $\|x_1\| > \|x_2\|$, we have $\|x_1\| \leq \|x_1+x_2\|$.
Hence we get a conclusion.

(2) It suffices to show that $v_1,\dots, v_n \in V$ are linearly independent if $\|v_1\|,\dots, \|v_n\|$ are all distinct.
We assume that $v_1,\dots,v_n \in V \setminus \{0\}$ are not linearly independent, that is, $a_1 v_1+\cdots+a_n v_n = 0$ for some $a_1,\dots,a_n \in K$.
We can assume that $a_i \neq 0$ for all $i$.
Since $K$ is trivially valued, we have $\|av\|=\|v\|$ for any $a \in K$ and $v \in V$.
Hence by (1), we have $0 = \|a_1 v_1 +\cdots+ a_n v_n\| = \max\{\|v_1\|,\dots,\|v_n\|\}$, which is a contradiction.
\end{proof}

We set
\[ \Fil^t(V,\nm) := \{ v \in V \,|\, \|v\| \leq e^{-t} \} \quad \text{for} \ t \in \R. \]
Remark that $\Fil^t(V,\nm)$ is a vector space over $K$ for any $t \in \R$ because $\ab$ is trivial.
Then $\{\Fil^t(V,\nm)\}_{t \in \R}$ satisfies the following conditions:

\begin{prop}
\label{rfil}
\begin{enumerate}
\item For sufficiently positive $t \in \R$, $\Fil^t(V,\nm) = \{0\}$.
\item For sufficiently negative $t \in \R$, $\Fil^t(V,\nm) = V$.
\item For any $t \geq s$, $\Fil^t(V,\nm) \subseteq \Fil^s(V,\nm)$.
\item The function $\R \ni t \mapsto \dim_K \Fil^t(V,\nm)$ is left-continuous.
\end{enumerate}
\end{prop}

\begin{proof}
(1) and (2) follow from Lemma \ref{norm1} and (3) and (4) follow from by definition.
\end{proof}

We set
\begin{equation*}
\lm(V,\nm) := \sup\{t \in \R \,|\, \Fil^t(V,\nm) \neq \{0\} \}.
\end{equation*}
By convention, $\lm(V,\nm) = -\infty$ if $V = \{0\}$.
By Proposition \ref{rfil}, we have $\lm(V,\nm) < +\infty$ and by Lemma \ref{norm1}, we can replace ``$\sup$'' by ``$\max$'' in the above definition.

\subsection{Berkovich space}
\label{BS}
Let $K$ be a field equipped with an absolute value $\ab$.
We assume that $K$ is complete with respect to $\ab$.
Let $X$ be a scheme over $\Spec K$.
We define the analytification of $X$ in the sence of Berkovich (for detail, see \cite{berkovich2012spectral}).

\begin{defi}
The \textit{analytification of $X$ in the sense of Berkovich}, or \textit{Berkovich space associated to $X$} is the set of pairs $x = (p,\ab_x)$ where $p \in X$ and $\ab_x$ is an absolute value on the residue field $\kappa(x) := \kappa(p)$ which is an extension of $\ab$, denoted by $X^\an$.
The map $j:X^\an \rightarrow X, (p,\ab_x) \mapsto p$ is called the \textit{specification map}.
\end{defi}

Let $U$ be a non-empty Zariski open subset of $X$.
The subset $U^\an := j^{-1}(U)$ of $X^\an$ is called a \textit{Zariski open subset of $X^\an$}.
A regular function $f \in \Ox_X(U)$ on $U$ define a function $|f|$ on $U^\an$ as follows:
\[ |f|(x) := |f(j(x))|_x \quad \text{for} \ x \in U^\an. \]
We also denote $|f|(x)$ by $|f|_x$.

We define a topology on $X^\an$ as the most coarse topology which makes $j$ and $|f|$ continuous for any Zariski open subset $U$ of $X$ and $f \in \Ox_X(U)$.
This is called the \textit{Berkovich topology}.
Remark that $X^\an$ is Hausdorff (resp. compact) if $X$ is separated (resp. proper) over $\Spec K$.

Let $f:X \rightarrow Y$ be a morphism of schemes over $\Spec K$.
There is a continuous map $f^\an:X^\an \rightarrow Y^\an$ such that the following diagram is commutative:
\[
\begin{tikzcd}
X \arrow{r}{f} \arrow{d}{j} & Y \arrow{d}{j} \\
X^\an \arrow{r}{f^\an} & Y^\an 
\end{tikzcd}
\]

Concretely, $f^\an$ is constructed as follows:
Let $x = (p,\ab_x) \in X^\an$ and $q = f(p) \in Y$.
We remark that $\kappa(y)=\kappa(q)$ is a subfield of $\kappa(x)=\kappa(p)$.
Then $y = f^\an(x)$ is given by $q=f(p)$ and the absolute value $\ab_y$ on $\kappa(q)$ which is the restriction of $\ab_x$.

In the following, $(K,\ab)$ is a trivially valued field.
For $x \in X$, let $x^\an = (x,\ab_0) \in X^\an$ where $\ab_0$ is the trivial absolute value on $\kappa(x)$.
This correspondence gives a section of $j$, which is denoted by $\sigma:X \rightarrow X^\an$.

Now we introduce an important subset of $X^\an$.
We assume that $X$ is normal projective variety over $\Spec K$.
Let $\eta \in X$ be the generic point of $X$ and $X^{(1)} = \{x \in X \,|\, \mathrm{codim}_X\overline{\{x\}} = 1\}$.
Let $K(X)$ be the function field of $X$.
Firstly, for $x \in X^{(1)}$, we set
\[ (\eta^\an,x^\an) := \left\{ \xi \in X^\an \,\middle|\, j(\xi)=\eta, \ab_\xi = e^{-t(\xi)\ord_x(\ndot)}\ \text{on}\ K(X), t(\xi) \in (0,+\infty)  \right\} \]
and
\[ [\eta^\an,x^\an] := \{\eta^\an\} \cup (\eta^\an,x^\an) \cup \{x^\an\}. \]
Then the correspondence $\xi \mapsto t(\xi)$, $\eta^\an \mapsto 0$ and $x^\an \mapsto +\infty$ gives a homeomorphism from $(\eta^\an,x^\an)$ (resp. $[\eta^\an,x^\an]$) to $(0,+\infty)$ (resp. $[0,+\infty]$).
Hence we sometimes identify $(\eta^\an,x^\an)$ (resp. $[\eta^\an,x^\an]$) with $(0,+\infty)$ (resp. $[0,+\infty]$).

We set $X_\dv^\an := \bigcup_{x \in X^{(1)}} [\eta^\an,x^\an]$.
Then we can illustrate $X_\dv^\an$ by an infinite tree as follows:

\begin{center}
\begin{tikzpicture}
\draw[thin] (0,1) node[above]{$\eta^\an$} -- (0,0) node[below]{$x^\an$};
\draw[thin] (0,1)--(-3,0);
\draw[thin] (0,1)--(-2,0);
\draw[thin] (0,1)--(1,0);
\draw[thin] (0,1)--(3,0);
\draw[dotted] (-1,0) node{$\cdots$};
\draw[dotted] (2,0) node{$\cdots$};
\filldraw (0,1) circle [radius=1.5pt];
\filldraw (0,0) circle [radius=1.5pt];
\filldraw (-3,0) circle [radius=1.5pt];
\filldraw (-2,0) circle [radius=1.5pt];
\filldraw (1,0) circle [radius=1.5pt];
\filldraw (3,0) circle [radius=1.5pt];
\end{tikzpicture}
\end{center}

We remark that $X_\dv^\an = X^\an$ if $\dim X = 1$.

\begin{lem}
\label{den}
$X_\dv^\an$ is dense in $X^\an$.
\end{lem}

\begin{proof}
For the proof, it is sufficient to show that, for any regular function $f$ on a Zariski open set $U$ in $X$ and any $x \in U^\an$, the value $|f|(x)$ is belonged to the closure $W$ of $\{ |f|(z) \,|\, z \in X_\mathrm{div}^\an \cap U^\an \} \subset \R_+$. 
If $f$ has no pole on $X$, then $f$ is regular on the whole $X$, so $f$ is a constant function and algebraic over $k$ because $X$ is normal and projective.
Therefore $|f|(z) = 1$ on $X^\an$, so it is clear that $|f|(x) \in W$.

We next assume that $f$ has poles on $X \setminus U$.
In this case, there are $y,y' \in X^{(1)}$ such that $f(y)=0$ and $f$ has a pole at $y'$ because $X$ is normal.
Then, $|f|(t) = e^{-at}$ for $t \in (\eta^\an,y^\an)$, $|f|(t') = e^{a't'}$ for $t' \in (\eta^\an,y'^\an)$ for some $a,a'>0$ and $|f|(\eta^\an) = 1$, which implies that $W=\R_+$ and we complete the proof. 
\end{proof}

Let $\R_{>0}$ be the multiplicative group of positive real numbers.
There is an action of $\R_{>0}$ to $X^\an$.
For $r \in \R_{>0}$ and $x=(p,\ab_x) \in X^\an$, we define
\[ r^*x := (p,\ab_x^r). \]
We also denote $r^*x$ by $x^r$.
This action is called the \textit{scaling action} in \cite{boucksom2018singular}.
The scaling action is free faithful and preserve the subset $[\eta^\an,x^\an]$ for all $x \in X^{(1)}$.

Finally, we introduce the reduction map $\red:X^\an \rightarrow X$.
For $x \in X^\an$, let $\widehat{\kappa}(x)$ be the completion of $\kappa(x)$ with respect to $\ab_x$ and we also denote the absolute value on $\widehat{\kappa}(x)$ by $\ab_x$.
We set $o_x := \{f \in \widehat{\kappa}(x) \,|\, |f|_x \leq 1 \}$ and $m_x := \{f \in \widehat{\kappa}(x) \,|\, |f|_x < 1 \}$.
Then $o_x$ is a local ring and $m_x$ is the maximal ideal of $o_x$.
If $\ab_x$ is trivial on $\kappa(x)$, then $o_x = \kappa(x)$ and $m_x = \{0\}$.
Let $p_x:\Spec \widehat{\kappa}(x) \rightarrow X$ be a $K$-morphism of schemes defined by $j(x)$ and $\iota_x:\Spec \widehat{\kappa}(x) \rightarrow \Spec o_x$ be a $K$-morphism defined by the inclusion $o_x \hookrightarrow \widehat{\kappa}(x)$.
By the valuation criterion of properness (for instance, see \cite{hartshorne2013algebraic}), there is a unique $K$-morphism $\phi_x:\Spec o_x \rightarrow X$ such that $p_x = \phi_x \circ \iota_x$.
\[
\begin{tikzcd}
\Spec \widehat{\kappa}(x) \arrow{r}{p_x} \arrow{d}{\iota_x} & X \arrow{d} \\
\Spec o_x \arrow[dashed]{ru}{\exists ! \phi_x} \arrow{r} & \Spec K
\end{tikzcd}
\]
Then we define $\red(x) \in X$ to be the image of $m_x$ by $\phi_x$.
The map $\red:X^\an \rightarrow X$ defined by the above correspondence is called the \textit{reduction map}.
The morphism $\phi_x$ induces a homomorphism $\Ox_{X,\red(x)} \rightarrow o_x$.
Hence we have
\begin{equation}
\label{red}
\forall f \in \Ox_{X,\red(x)}, |f|_x \leq 1.
\end{equation}
We remark that $j \neq \red$.
For example, for any $x \in X$, $\red(x^\an) = x$ and for any $\xi \in (\eta^\an,x^\an)$, $\red(\xi) = x$.
It is known that $\red:X^\an \rightarrow X$ is anti-continuous, that is, for any open set $U$ of $X$, $\red^{-1}(U)$ is closed in $X^\an$.

\section{Adelic $\R$-Cartier divisors over a trivially valued field}
In this section, we study fundamental properties of Arakelov geometry over a trivially valued field.
Throughout this section, let $K$ be a trivially valued field, $X$ be a normal projective variety over $\Spec K$ and $X^\an$ be the analytification of $X$ in the sense of Berkovich.
Let $K(X)$ be the function field of $X$.

\subsection{Green functions}
Let $U^\an$ be a non-empty Zariski open subset of $X^\an$.
We denote by $C^0(U^\an)$ the set of continuous functions on $U^\an$.
We define
\[ \widehat{C}^0(X^\an) := \varinjlim_{\substack{\text{non-empty} \\ \text{Zariski open} \\ \text{subset of} \ X^\an}} C^0(U^\an). \]
Then $C^0(U^\an)$ and $\widehat{C}^0(X^\an)$ are $\R$-algebras and we have a canonical homomorphism $C^0(U^\an) \rightarrow \widehat{C}^0(X^\an)$.
Since $U^\an$ is dense in $X^\an$, this homomorphism is injective.
Hence we sometimes identify a function in $C^0(U^\an)$ with a function in $\widehat{C}^0(X^\an)$.
We say that a function in $\widehat{C}^0(X^\an)$ \textit{extends to a continuous function on $U^\an$} if it is in the image of the canonical injection $C^0(U^\an) \rightarrow \widehat{C}^0(X^\an)$.

\begin{defi}
Let $D$ be an $\R$-Cartier divisor on $X$.
We say that a function $g \in \widehat{C}^0(X^\an)$ is a \textit{$D$-Green function of $C^0$-type} (or simply a \textit{Green function of $D$}) if for any non-empty Zariski open subset $U$ of $X$ and any local equation $f \in K(X)_\R^\times$ of $D$ on $U$, the function $g + \log|f|$ extends to a continuous function on $U^\an$.
\end{defi}

\begin{eg}
Let $\Ps_K^n = \mathrm{Proj}\,K[T_0,\dots,T_n]$ be the $n$-dimensional projective space.
We set $z_i = T_i/T_0$ for $i=0,\dots,n$ and $D = \{ T_0=0 \}$.
Then $g=\log\max\{a_0,a_1|z_1|,\dots,a_n|z_n|\} \in \widehat{C}^0(\Ps_K^{n,\an})$, where $a_0,\dots,a_n \in \R_{>0}$, is a $D$-Green function of $C^0$-type.
\end{eg}

\begin{prop}
\label{pg}
Let $D,D'$ be $\R$-Cartier divisors on $X$ and $g,g' \in \widehat{C}^0(X^\an)$ be Green function of $D,D'$ respectively.
\begin{enumerate}
\item For any $s \in K(X)_\R^\times$, $-\log|s| \in \widehat{C}^0(X^\an)$ is a Green function of $(s)$.
\item For any $a,a' \in \R$, $ag + a'g'$ is a Green function of $aD+a'D'$.
\item If $D$ is the zero divisor, a $D$-Green function of $C^0$-type coincides to a continuous function on $X^\an$.
\item Let $\pi :Y \rightarrow X$ be a morphism of projevtive varieties over $K$ such that $\pi(Y) \nsubseteq \mathrm{Supp}D$.
Then $\pi^*g = g \circ \pi^\an$ is a Green function of $\pi^*D$.
\end{enumerate}
\end{prop}

\begin{proof}
(1) It follows from that $s$ is a local equation of $(s)$ on any Zariski open subset.

(2) Let $U$ be a non-empty Zariski open subset of $X$, $f,f'$ be local equations of $D,D'$ on $U$ respectively.
Then $f^a{f'}^{a'}$ is a local equation of $aD+a'D'$ on $U$ and $(ag+a'g')+(a\log|f|+a'\log|f'|)=a(g+\log|f|)+a'(g'+\log|f'|)$ extends a continuous function on $U^\an$.

(3) It follows from (2).

(4) Let $U$ be a non-empty Zariski open subset of $X$, $f$ be local equations of $D$ on $U$.
Then $\pi^*f = f \circ \pi$ is a local equation of $\pi^*D$ on $\pi^{-1}(U)$ and $\pi^*g + \log|\pi^*f| = (g + \log|f|) \circ \pi^\an$ extends a continuous function on $(\pi^\an)^{-1}(U^\an)$ because $\pi^\an$ is continuous.
\end{proof}

\begin{prop}[{c.f. \cite[Proposition 2.5]{chen2017sufficient}}]
\label{eg}
For any $\R$-Cartier divisor $D$, there exists a $D$-Green function of $C^0$-type.
\end{prop}

\begin{proof}
Firstly, we assume that $D$ is an ample Cartier divisor.
Let $m$ be a positive integer such that $mD$ is very ample.
Then we have an closed immersion $\pi:X \hookrightarrow \Ps_K^n = \mathrm{Proj}K[T_0,\dots,T_n]$ such that $\Ox_X(mD) = \pi^*\Ox(1)$.
We set $z_i = T_i/T_0$ for $i=0,\dots,n$ and $D_0 = \{ T_0=0 \}$.
Then $g_0 = \log\max\{1,|z_1|,\dots,|z_n|\}$ is a $D_0$-Green function of $C^0$-type.
By Proposition \ref{pg} (4), $\pi^*g_0$ is a $\pi^*D_0$-Green function of $C^0$-type.
Since $mD$ and $\pi^*D_0$ are linearly equivalent, there is a non-zero rational function $s \in K(X)^\times$ such that $mD = \pi^*D_0 + (s)$.
Then $(\pi^*g_0 - \log|s|)/m$ gives a $D$-Green function of $C^0$-type.

Next, we assume that $D$ is a Cartier divisor.
Then we can write $D$ as $A-A'$ where $A,A'$ are ample divisors.
From the previous discussion, there are Green functions $g_A,g_{A'}$ of $A,A'$ respectively.
Then $g_A-g_{A'}$ gives a $D$-Green function of $C^0$-type.

In general, there are Cartier divisors $D_1,\dots,D_n$ and $a_1,\dots,a_n \in \R$ such that $D=a_1D_1+\cdots+a_nD_n$.
Let $g_i$ be a $D_i$-Green function of $C^0$-type for $1,\dots,n$.
Then $a_1g_1+\cdots+a_ng_n$ gives a $D$-Green function of $C^0$-type.
\end{proof}

\begin{prop}[{c.f. \cite[Proposition 2.6]{chen2017sufficient}}]
\label{eff}
Let $D$ be an effective $\R$-Cartier divisor on $X$ and $g$ be a $D$-Green function of $C^0$-type.
Then the function $e^{-g} \in \widehat{C}^0(X^\an)$ extends to a non-negative continuous function on $X^\an$.
\end{prop}

\begin{proof}
Let $U$ be a non-empty Zariski open subset of $X$ and $f$ be a local equation of $D$ on $U$.
Since $g+\log|f|$ extends a continuous function on $U^\an$, $e^{-g} = |f| \ndot e^{-(g+\log|f|)}$ extends a non-negative continuous function on $U^\an$.
We remark that $|f| \in C^0(U^\an)$ because $D$ is effective.
By gluing continuous functions, $e^{-g}$ extends a non-negative continuous function on $X^\an$.
\end{proof}

By Proposition \ref{eff}, we sometimes consider a Green function of an effective $\R$-Cartier divisor as a map $X^\an \rightarrow \R \cup \{+\infty\}$.

\subsection{Continuous metrics on an invertible $\Ox_X$-module}
\begin{defi}
Let $\Lb$ be an invertible $\Ox_X$-module.
We say that a family $\varphi = \{\ab_\varphi(x)\}_{x \in X^\an}$ is a \textit{metric on $\Lb$} if $\ab_\varphi(x)$ is a norm on $\Lb(x) := \Lb \otimes_{\Ox_X} \widehat{\kappa}(x)$ for all $x \in X^\an$.
A metric $\varphi = \{\ab_\varphi(x)\}_{x \in X^\an}$ on $\Lb$ is \textit{continuous} if for any Zariski open subset $U$ of $X$ and non-zero section $s \in H^0(U,\Lb)\setminus\{0\}$, $|s|_\varphi(x)$ is a continuous function on $U^\an$.
\end{defi}

\begin{defi}
Let $\Lb$ be an invertible $\Ox_X$-module on $X$ and $\varphi$ be a continuous metric on $\Lb$.
Then we define a norm $\nm_\varphi$ on $H^0(X,\Lb)$ by
\[ \|s\|_\varphi := \sup_{x \in X^\an} |s|_\varphi(x) \quad \text{for} \ s \in H^0(X,\Lb).  \]
\end{defi}

\begin{defi}
Let $\Lb, \Lb'$ be invertible $\Ox_X$-modules.
Let $\varphi = \{\ab_\varphi(x)\}_{x \in X^\an}$ and $\varphi' = \{\ab_{\varphi'}(x)\}_{x \in X^\an}$ be metrics of $\Lb,\Lb'$ respectively.
We define the metric $\varphi + \varphi'$ of $\Lb \otimes \Lb'$ by
\[ |s \otimes s'|_{\varphi+\varphi'}(x) = |s|_\varphi(x) \ndot |s|_{\varphi'}(x) \]
for $x \in X^\an$, $s \in \Lb(x)$ and $s' \in \Lb'(x)$.
We set the \textit{dual metric $-\varphi$ of $\varphi$ on $\Lb^\vee$} as
\[ |\alpha(s)|_x = |\alpha|_{-\varphi}(x) \ndot |s|_\varphi(x) \]
for $x \in X^\an$, $s \in \Lb(x)$ and $\alpha \in \Lb^\vee(x) = (\Lb(x))^\vee$.
If $\varphi$ and $\varphi'$ are continuous, $\varphi + \varphi'$ and $-\varphi$ are also continuous by definition.
\end{defi}

We see the relation between Green functions and continuous metrics.
Let $\Lb$ be an invertible $\Ox_X$-module and $\varphi = \{\ab_\varphi(x)\}_{x \in X^\an}$ be a continuous metric on $\Lb$.
Let $s$ be a non-zero rational section of $\Lb$.
Then $(\mathrm{div}(s), -\log|s|_\varphi)$ is a $\mathrm{div}(s)$-Green function of $C^0$-type.
In fact, let $U$ be a non-empty Zariski open subset of $X$ which trivialize $\Lb$ and $\omega$ be a local basis of $\Lb$ on $U$.
Then $s$ is denoted by $f\omega$ for some $f \in K(X)$.
Since $\mathrm{div}(s)$ is defined by $f$ on $U$, we have
\[ -\log|s|_\varphi + \log|f| = -\log|f| - \log|\omega|_\varphi + \log|f| =  -\log|\omega|_\varphi, \]
which is a continuous function on $U^\an$.

Conversely, let $D$ be a Cartier divisor and $g$ be a $D$-Green function of $C^0$-type.
Then we equip $\Ox_X(D)$ with a continuous metric $\varphi_g = \{\ab_g(x)\}_{x \in X^\an}$ as follows:
Let $U$ be a non-empty Zariski open subset of $X$ and $f$ be a local equation of $D$ on $U$.
Since $1/f$ is a local basis of $\Ox_X(D)$ on $U$, we can denote any section $s \in \Ox_X(D)(U)$ by $a/f$ for some $a \in \Ox_X(U)$.
Then we define $|s|_g(x) := |a|_x \ndot e^{-g(x) - \log|f|_x}$ for $x \in U^\an$.

By Proposition \ref{pg}, a continuous metric $\varphi$ on $\Ox_X$ corresponds to a continuous function $g_\varphi$ on $X^\an$.
So for continuous metrics $\varphi$ and $\psi$ on an invertibel $\Ox_X$-module, we write $\varphi \geq \psi$ if the continuous function $g_{\varphi-\psi}$, which corresponds to $\varphi-\psi$, is non-negative on $X^\an$.

\subsection{Adelic $\R$-Cartier divisors}

Let $\K = \Q,\R$ or a blank symbol.

\begin{defi}
We say that a pair $\overline{D} = (D,g)$ is an \textit{adelic $\K$-Cartier divisor} if $D$ is an $\K$-Cartier divisor and $g$ is a $D$-Green function of $C^0$-type.
We denote by $\widehat{\Div}(X)_\K$ the set of $\K$-Cartier divisors.
Remark that $\widehat{\Div}(X)_{\R} \ncong \widehat{\Div}(X) \otimes_{\Z} \R$.
A non-zero $\K$-rational function $f \in K(X)_\K^\times$ naturally gives an adelic $\K$-Cartier divisor $((f),-\log|f|)$, which is called a \textit{$\K$-principal} and denoted by $\widehat{(f)}$.
We say that two adelic $\R$-Cartier divisors $\overline{D}_1,\overline{D}_2$ are \textit{$\K$-linearly equivalent} if $\overline{D}_1-\overline{D}_2$ is $\K$-principal.
Let $\widehat{\Pic}(X)$ be $\widehat{\Div}(X)$ modulo linearly equivalence and it is called the \textit{arithmetic Picard group}. 
An adelic $\K$-Cartier divisor $(D,g)$ is \textit{effective} if $D$ is effective and $g$ is a non-negative. 
Then we write $(D,g) \geq 0$.
\end{defi}

Let $\overline{D} = (D,g)$ be an adelic $\R$-Cartier divisor on $X$.
Then the set of ``global sections'' $H^0(D)$ is given by
\[ H^0(D)  = \{f \in K(X)^\times \,|\, D + (f) \geq 0 \} \cup \{0\}. \]
Let $s \in H^0(D) \setminus \{0\}$.
By Proposition \ref{eff}, the function $|s| e^{-g} = e^{-g+\log|s|}$ extends to a non-negative function on $X^\an$.
We denote this function by $|s|_g:X^\an \rightarrow \R_+$.
Then we define
\[ \|s\|_g := \sup_{x \in X^\an} |s|_g(x). \]
The map $\nm_g :H^0(D) \rightarrow \R_+$ gives an ultrametric norm on $H^0(D)$ over $K$ and it coincides with the supremum norm induced by the continuous metric on $\Ox_X(D)$ corresponding to $g$.
Moreover, $\bigoplus_{n=0}^\infty (H^0(nD),\nm_{ng})$ is a normed graded ring over $K$ by definition.

We set
\begin{equation*}
\lm(D,g) := \lm(H^0(D),\nm_g),
\end{equation*}
and
\[ \lmasy(D,g) := \limsup_{n \rightarrow +\infty} \frac{1}{n}\lm(nD,ng). \] 
Since $\bigoplus_{n=0}^\infty (H^0(nD),\nm_{ng})$ is a normed graded ring, the sequence $\{\lm(nD,ng)\}_n$ is super-additive, that is,
\[ \lm((m+n)D,(m+n)g) \geq \lm(mD,mg) + \lm(nD,ng) \quad \text{for}\ \forall m,n \in \Z_+. \]
Hence by Fekete's lemma, we have
\[ \lmasy(D,g) = \lim_{n \rightarrow +\infty} \frac{1}{n}\lm(nD,ng) = \sup_{n \geq 1} \frac{1}{n}\lm(nD,ng). \]
Later, we will show that $\lmasy(D,g) < +\infty$.

\begin{defi}
Let $(D,g)$ be an adelic $\R$-Cartier divisor.
We say that a non-zero global section $s \in H^0(D) \setminus \{0\}$ is a \textit{small section} if $\|s\|_g \leq 1$ or equivalently $s \in \Fil^0(H^0(D),\nm_g)$.
Moreover, if $\|s\| < 1$, it is called a \textit{strictly small section}.
\end{defi}

\begin{prop}
Let $\overline{D}=(D,g)$ be an adelic $\R$-Cartier divisor on $X$.
Then we have
\[ \Fil^0(H^0(D),\nm_g) = \left\{ s \in K(X)^\times \, \middle| \, \overline{D} + \widehat{(s)} \geq 0 \right\} \cup \{0\}. \]
\end{prop}

\begin{proof}
Let $s \in H^0(D) \setminus \{0\}$.
By definition,
\begin{equation*}
\begin{aligned}
\|s\|_g \leq 1 &\Leftrightarrow e^{-g+\log|s|} \leq 1\ \text{on}\ X^\an \\
&\Leftrightarrow g-\log|s| \geq 0 \ \text{on}\ X^\an.
\end{aligned}
\end{equation*}
\end{proof}

Small sections play the similar role as global sections in algebraic geometry.
Therefore we are interested in the asymptotic behavior of $\Fil^0(H^0(D),\nm_{ng})$ as $n \rightarrow +\infty$.

\subsection{Associated $\R$-Weil divisors}

\begin{defi}
Let $(D,g)$ be an adelic $\R$-Cartier divisor on $X$.
For any $x \in X^{(1)}$,
\[ \mu_x(g) := \inf_{\xi \in (\eta^\an,x^\an)} \frac{g(\xi)}{t(\xi)} \in \R \cup \{-\infty\}. \]
Clearly $\mu_x(g) \geq 0$ if and only if $g \geq 0$ on $(\eta^\an,x^\an)$.
Moreover $\mu_x(g) = -\infty$ if and only if $g(\eta^\an) < 0$, which implies that if $\mu_x(g) = -\infty$ for some $x \in X^{(1)}$, then $\mu_x(g) = -\infty$ for every $x \in X^{(1)}$.
\end{defi}

The above invariant $\mu_x(g)$ has following properties:

\begin{prop}[{c.f. \cite[Proposition 5.7]{chen2017sufficient}}]
\label{mu0}
Let $(D,g)$ be an adelic $\R$-Cartier divisor on $X$.
For all but finitely many $x \in X^{(1)}$, we have $\mu_x(g) \leq 0$.
\end{prop}

\begin{proof}
Let $U$ be a non-empty Zariski open subset of $X$ such that $g$ is a continuous function on $U^\an$.
Then $g$ is continuous on $[\eta^\an,x^\an]$ for all $x \in U \cap X^{(1)}$.
Since $[\eta^\an,x^\an]$ is compact, $g|_{[\eta^\an,x^\an]}$ is bounded above.
Hence we have $\mu_x(g) \leq 0$ for for all $x \in U \cap X^{(1)}$, which implies the assertion because $X^{(1)} \setminus U$ is a finite set.
\end{proof}

\begin{prop}[{c.f. \cite[Lemma 5.8]{chen2017sufficient}}]
\label{mu1}
Let $(D,g)$ be an adelic $\R$-Cartier divisor on $X$ and $x \in X^{(1)}$.
\begin{enumerate}
\item For any $s \in K(X)_\R^\times$, we have
\[ \mu_x(g - \log |s|) = \mu_x(g) + \mathrm{ord}_x(s). \]
\item We have $\mu_x(g) \leq \mathrm{ord}_x(D)$.
\end{enumerate}
\end{prop}

\begin{proof}
(1) By definition of $X_\dv^\an$, for any $s \in K(X)_\R^\times$, we have
\[ -\log|s|(\xi) = t(\xi)\ord_x(s),  \quad \xi \in (\eta^\an,x^\an). \]
Hence we obtain that
\[ \mu_x(g-\log|s|) = \inf_{\xi \in (\eta^\an,x^\an)} \frac{g(\xi)-\log|s|(\xi)}{t(\xi)} = \mu_x(g) + \ord_x(s). \]

(2) Let $f \in K(X)_\R^\times$ be a local equation of $D$ around $x$.
Then $g + \log|f|$ extends to a continuous function on $[\eta^\an,x^\an]$.
Since $(g + \log|f|)|_{[\eta^\an,x^\an]}$ is bounded above, we have $\mu_x(g + \log|f|) \leq 0$.
By (1), we get $\mu_x(g) \leq \ord_x(D)$.
\end{proof}

Now we introduce an important divisor.

\begin{defi}
Let $(D,g)$ be an adelic $\R$-Cartier divisor on $X$.
We say that $(D,g)$ is \textit{$\mu$-finite} if $\mu_x(g) = 0$ for all but finitely many $x \in X^{(1)}$, which is equivalent to $\mu_x(g) \geq 0$ for all but finitely many $x \in X^{(1)}$ by Proposition \ref{mu0}.
If $(D,g)$ is $\mu$-finite, we can define an $\R$-Weil divisor on $X$ as follows:
\[ D_{\mu(g)} := \sum_{x \in X^{(1)}} \mu_x(g) [x]. \]
It is called an \textit{$\R$-Weil divisor associated with $(D,g)$}.
Remark that $D_{\mu(g)}$ may not be an $\R$-Cartier divisor.
\end{defi}

For example, if $(D,g)$ has a Dirichlet property (which means that $(D,g)$ is $\R$-linearly equivalent to an effective adelic $\R$-Cartier divisor), then $(D,g)$ is $\mu$-finite.

By Proposition \ref{mu1}, we have $D_{\mu(g)} \leq D$ and 
\begin{equation}
\label{mulin}
{(D + (s))}_{\mu(g - \log |s|)} = D_{\mu(g)} + (s).
\end{equation}

\begin{prop}
Let $(D,g)$ be a $\mu$-finite adelic $\R$-Cartier divisor on $X$.
Then $(D,g)$ is effective if and only if $D_{\mu(g)}$ is effective.
\end{prop}

\begin{proof}
We first assume that $(D,g)$ is effective.
Then $g$ is non-negative on $X^{\an}$, so $\mu_x(g) \geq 0$ for any $x \in X^{(1)}$,
which implies $D_{\mu(g)}$ is effective.

Conversely we assume that $D_{\mu(g)}$ is effective.
Then $g$ is non-negative on  $X_{\mathrm{div}}^{\an}$, but $X_{\mathrm{div}}^{\an}$ is dense in $X^{\an}$ by Lemma \ref{den}, so it follows that $g$ is non-negative on the whole $X^{\an}$.
Moreover by Proposition \ref{mu1}, we have
\[ \mathrm{ord}_x(D) \geq \mu_x(g) \geq 0 \]
for any $x \in X^{(1)}$, which completes the proof.
\end{proof} 

By the above proposition and the equation (\ref{mulin}), we have the following corollary:

\begin{cor}
\label{as1}
\[ H^0(D_{\mu(g)}) = \mathcal{F}^0(H^0(D),\nm_g) = \{ s \in H^0(D) \,|\, \|s\|_g \leq 1 \}. \]
\end{cor}

\subsection{Canonical Green function}
For any $\R$-Cartier divisor $D$, we can naturally give a $D$-Green function of $C^0$-type as follows:
For any $x \in X^\an$, let $f \in K(X)_\R^\times$ be a local equation of $D$ around $\red(x) \in X$.
Then we define
\[ g_D^\mathrm{c}(x) := -\log|f|_x. \]
This definition is independent of the choice of a local equation.
In fact, let $f' \in K(X)_\R^\times$ be another local equation.
Then there is an element $a \in (\Ox_{X,\red(x)})_\R^\times$ such that $f' = af$.
Since $|a|_x = 1$ by (\ref{red}), we have $-\log|f'|_x = -\log|f|_x$.

\begin{prop}
\label{cg}
The function $g_D^\mathrm{c}$ is a $D$-Green function of $C^0$-type.
\end{prop}

\begin{proof}
It is enough to show that for any non-empty Zariski open subset $U$ of $X$ and local equation $f$ of $D$ on $U$, $g_D^\mathrm{c} + \log|f|$ extends to a continuous function on $U^\an$.
Let $x \in U^\an$.
If $\red(x) \in U$, then $g_D^\mathrm{c}(x) = -\log|f|_x$.
Hence we have $g_D^\mathrm{c}(x) + \log|f|_x = 0$.
Next, we assume that $\red(x) \notin U$.
Let $U'$ be a non-empty Zariski open neighborhood of $\red(x)$ and $f'$ be a local equation of $D$ on $U'$.
Then we have $g_D^\mathrm{c}(x) = -\log|f'|_x$
We remark that $j(x) \in U'$, hence $U \cap U' \neq \emptyset$.
There is a non-zero regular function $u \in (\Ox_X(U \cap U'))_\R^\times$ such that $f' = uf$ on $U \cap U'$.
Therefore we obtain that $g_D^\mathrm{c}(x) + \log|f|_x = -\log|u|_x$, which is continuous on $U^\an \cap {U'}^\an$.
Finally, let $y \in U^\an \cap {U'}^\an$ such that $\red(y) \in U$.
Since $u \in (\Ox_{X,\red(y)})_\R^\times$, we have $|u|_y = 1$ by (\ref{red}).
Hence $g_D^\mathrm{c}(y) + \log|f|_y = - \log|u|_y = 0$, which completes the proof.
\end{proof}

\begin{rem}
Proposition \ref{cg} gives the another proof of Proposition \ref{eg}.
\end{rem}

\begin{defi}
The function $g_D^\mathrm{c}$ is called the \textit{canonical Green function of $D$}.
\end{defi}

\begin{prop}
\label{cgp}
\begin{enumerate}
\item For any $s \in K(X)_\R^\times$, $g_{(s)}^\mathrm{c} = -\log|s|$.
\item For any $D,D' \in \Div(X)_\R$ and $a,a' \in \R$, $g_{aD+a'D'}^\mathrm{c} = ag_D^\mathrm{c} + a'g_{D'}^\mathrm{c}$.
\end{enumerate}
\end{prop}

\begin{proof}
(1) Since $(s)$ is globally defined by $s$, it follows by definition of the canonical Green function.

(2) Let $x \in X^\an$ and $f,f'$ be local equations of $D,D'$ around $\red(x)$ respectively.
Then $f^a{f'}^{a'}$ is a local equation of $aD+a'D'$ around $\red(x)$.
Hence we have
\[ g_{aD+a'D'}^\mathrm{c}(x) = -\log|f^a{f'}^{a'}|_x = -a\log|f|_x-a'\log|f'|_x = ag_D^\mathrm{c}(x)+a'g_{D'}^\mathrm{c}(x). \]
\end{proof}

Using the canonical Green function, we can define the following injective homomorphism:
\[ \phi:\Div(X) \rightarrow \widehat{\Div}(X), \quad D \mapsto (D,g_D^\mathrm{c}). \]
By Proposition \ref{cgp}, it induces an injective homomorphism $\overline{\phi}:\Pic(X) \rightarrow \widehat{\Pic}(X)$ such that the following diagram is commutative:
\[
\begin{tikzcd}
\Div(X) \arrow{r}{\phi} \arrow{d} & \widehat{\Div}(X) \arrow{d} \\
\Pic(X) \arrow{r}{\overline{\phi}} & \widehat{\Pic}(X).
\end{tikzcd}
\]

\subsection{Height function}
Here we see the hight function on $X^\an$ associated with an adelic $\R$-Cartier divisor, which is introduced by Chen and Moriwaki \cite{chen2017sufficient}.

\begin{defi}
Let $(D,g)$ be an adelic $\R$-Cartier divisor on $X$.
We set $h_{(D,g)}^\an := g - g_D^\mathrm{c}$, which is called \textit{hight function on $X^\an$ associated with $(D,g)$}.
\end{defi}

\begin{prop}[{c.f. \cite[Proposition 4.3]{chen2017sufficient}}]
\label{hp}
Let $\overline{D},\overline{D}'$ be adelic $\R$-Cartier divisors on $X$.
\begin{enumerate}
\item For any $s \in K(X)_\R^\times$, $h_{\widehat{(s)}}^\an = 0$ on $X^\an$.
\item For any $a,a' \in \R$, $h_{a\overline{D}+a'\overline{D}'}^\an = ah_{\overline{D}}^\an+a'h_{\overline{D}'}^\an$ on $X^\an$.
\end{enumerate}
\end{prop}

\begin{proof}
It immediately follow from Proposition \ref{cgp}.
\end{proof}

For any adelic $\R$-Cartier divisor $(D,g)$ on $X$, $h_{(D,g)}^\an$ is a continuous function on $X^\an$.
Hence we have the following homomorphism:
\[ \psi:\widehat{\Div}(X) \rightarrow C^0(X^\an), \quad (D,g) \mapsto h_{(D,g)}^\an. \]
This homomorphism is surjective.
By Proposition \ref{hp}, it induces a surjective homomorphism $\overline{\psi}:\widehat{\Pic}(X) \rightarrow C^0(X^\an)$ such that the following diagram is commutative:
\[
\begin{tikzcd}
\widehat{\Div}(X) \arrow{r}{\psi} \arrow{d} & C^0(X^\an) \\
\widehat{\Pic}(X) \arrow{ru}{\overline{\psi}}
\end{tikzcd}
\]

\begin{thm}
The following sequence is exact:
\[
\begin{tikzcd}
0 \arrow{r} & \Pic(X) \arrow{r}{\overline{\phi}} & \widehat{\Pic}(X) \arrow{r}{\overline{\psi}} & C^0(X^\an) \arrow{r} & 0.
\end{tikzcd}
\]
In particular, $\widehat{\Pic}(X) \simeq \Pic(X) \oplus C^0(X^\an)$.
\end{thm}

\begin{proof}
Since $\psi \circ \phi = 0$ by definition, we have $\overline{\psi} \circ \overline{\phi} = 0$.
Let $(D,g) \in \widehat{\Div}(X)$ such that $\overline{\psi}(D,g) = 0$.
Then there are $H \in \Div(X)$ and $s \in K(X)$ such that $(D,g) = (H,g_H^\mathrm{c}) + (s,-\log|s|)$.
By Proposition \ref{cgp}, we have $g = g_H^\mathrm{c} - \log|s| = g_D^\mathrm{c}$, which implies that $(D,g) = \phi(D)$.
Hence we obtain that $\mathrm{Im}\, \overline{\phi} = \mathrm{Ker}\, \overline{\psi}$.
\end{proof}

For $X = \Spec K$, the Berkovich space $X^\an$ associated with $X$ is a single point.
Hence we have $C^0(X^\an) = \R$, which implies that
\[ \widehat{\Pic}(X) \simeq \R. \]
This result corresponds to the fact that the Picard group of $\Ps^1$ is isomorphic to $\Z$ and the arithmetic Picard group of $\Spec \Z$ is isomorphic to $\R$.

\subsection{Scaling action for Green functions}

We saw that the multiplicative group $\R_{>0}$ acts $X^\an$ (see Section \ref{BS}).
Here we see that it also acts the set of $D$-Green functions of $C^0$-type.
Let $(D,g)$ be an adelic $\R$-Cartier divisor on $X$.
For $r \in \R_{>0}$, we define
\[ r^*g(x) := rg(x^{1/r}). \]
The function $r^*g$ is also $D$-Green function of $C^0$-type.
In fact, let $U$ be a non-empty Zariski open subset of $X$ and $f$ be a local equation of $D$ on $U$.
Then we have
\[ r^*g(x) + \log|f|_x = rg(x^{1/r}) + r\log|f|_x^{1/r} = r(g(x^{1/r}) + \log|f|_{x^{1/r}}) \]
on $U^\an$, which is continuous.
This action is also called the \textit{sacling action}.

\begin{prop}
\label{sag}
Let $D$ be an $\R$-Cartier divisor on $X$, $g,g'$ be $D$-Green functions of $C^0$-type.
\begin{enumerate}
\item The scaling action is linear, that is, for $r \in \R_{>0}$,
\[ r^*(g + g') = r^*g + r^*g'. \]
\item The scaling action preserves the canonical Green function $g_D^\mathrm{c}$.
\end{enumerate}
\end{prop}

\begin{proof}
(1) It is clear by definition.

(2) Let $x \in X^\an$ and $f$ be a local equation of $D$ around $\red(x)$.
Then $g_D^\mathrm{c}(x) = -\log|f|_x$.
Hence we have
\[ r^*g_D^\mathrm{c}(x) = -r\log|f|_{x^{1/r}} = -r\log|f|_x^{1/r} = -\log|f|_x = g_D^\mathrm{c}(x) \]
for any $r \in \R_{>0}$.
\end{proof}

\section{Arithmetic volume}
In this section, we introduce the arithmetic volume function and the bigness of adelic $\R$-Cartier divisors.
And we study some properties of the arithmetic volume function.
Throughout this section, let $K$ be a trivially valued field and $X$ be a normal projective variety over $\Spec K$.

\subsection{Big adelic $\R$-Cartier divisors}
We introduce the counterparts of $h^0(D)$ and $\vol(\ndot)$ in Arakelov geometry, which is given by Chen and Moriwaki \cite{chen2017sufficient}.
We set
\[ \adeg_+(D,g) := \int_0^{+\infty} \dim_K \Fil^t(H^0(D),\nm_g) \ \mathrm{d}t, \]
and
\[ \avol(D,g) := \limsup_{n \rightarrow +\infty} \frac{\adeg_+(nD,ng)}{n^{d+1}/(d+1)!}, \]
where $d = \dim X$.

\begin{defi}
We say that an adelic $\R$-Cartier divisor $(D,g)$ is \textit{big} if $\avol(D,g) > 0$.
\end{defi}

\begin{prop}
\label{mbig}
Let $(D,g)$ be an adelic $\R$-Cartier divisor on $X$.
If $(D,g)$ is big, then $(D,g)$ is $\mu$-finite, $D_{\mu(g)}$ is big and $\lmasy(D,g) > 0$.
In particular, $D$ is big.
\end{prop}

\begin{proof}
If $(D,g)$ is big, $(D,g)$ is $\Q$-linearly equivalent to an effective adelic $\R$-Cartier divisor, which implies that $(D,g)$ is $\mu$-finite.
By definition, for any integer $n > 0$, we have
\[ \adeg_+(nD,ng) \leq \dim_k\Fil^0(H^0(nD),\nm_{ng}){\max} \{\lm(nD,ng),0 \}. \]
Therefore we have
\[ \avol(D,g) \leq (d+1)\vol(D_{\mu(g)}){\max} \{\lmasy(D,g),0 \}, \]
by corollary \ref{as1}.
Since $\avol(D,g) > 0$, we have $\vol(D_{\mu(g)}) > 0$ and $\lmasy(D,g) > 0$.
\end{proof}

\subsection{Existence of limit of the arithmetic volume}

Firstly, we define
\[ \nu_{\max}(D,g) := \sup \{ t \in \R \,|\, (D,g-t) \ \text{is $\mu$-finite} \}. \]

\begin{lem}
Let $(D,g)$ be an adelic $\R$-Cartier divisor on $X$.
We have
\[ \lmasy(D,g) \leq \nu_{\max}(D,g) \leq g(\eta^\an). \]
\end{lem}

\begin{proof}
Clearly we can assume $\lmasy(D,g) \in \R$.
For a sufficiently large integer $n > 0$, there is a non-zero element $s \in H^0(nD) \setminus \{0\}$ such that $\|s\|_{ng} \leq e^{-\lm(nD,ng)}$, which is equivalent to $\|s\|_{ng-\lm(nD,ng)} \leq 1$.
Therefore $(D,g-\lm(nD,ng)/n)$ is effecive, which implies
\[ \frac{1}{n}\lm(nD,ng) \leq \nu_{\max}(D,g). \]
Taking a supremum, we get $\lmasy(D,g) \leq \nu_{\max}(D,g)$.

Next we show $\nu_{\max}(D,g) \leq g(\eta^\an)$.
For any $\epsilon > 0$, $g(\zeta) - (g(\eta^\an) + \epsilon)$ is negative around $\zeta = \eta^\an$.
So we have $\mu_x(g-(g(\eta^\an) + \epsilon)) = -\infty$ for any $x \in X^{(1)}$, which implies that $(D,g-(g(\eta^\an) + \epsilon))$ is not $\mu$-finite and $\nu_{\max}(D,g) \leq g(\eta^\an) + \epsilon$.
Since $\epsilon$ is arbitrary, we conclude that $\nu_{\max}(D,g) \leq g(\eta^\an)$.
\end{proof}

\begin{rem}
The above inequality is sometimes strict.
For example, let $X = \Ps_K^1 = \mathrm{Proj}K[T_0,T_1]$, $z=T_1/T_0$, $D = \{T_0=0\}$ and $x_\infty = (0:1)$.
Let $g_1 = 2\log\max\{2,|z|\} - \log\max\{1,|z|\}$.
Then $g_1(\xi) = 2\log2$ for $\xi \in [\eta^\an,x^\an]$ for $x \neq x_\infty$ and
\[ g_1(\xi) = 
\begin{cases}
2\log2 - \xi & (0 \leq \xi \leq \log2) \\
\xi & (\log2 \leq \xi),
\end{cases}
\]
on $[\eta^\an,x_\infty^\an]$.
Hence we have
\[ \mu_x(g_1-t) =
\begin{cases}
0 & (x \neq x_\infty) \\
\log2 - t & (x = x_\infty),
\end{cases}
\]
for $t \leq 2\log2$ and $\mu_x(g_1-t) = -\infty$ for $t > 2\log2$ and all closed point $x$ of $X$.
Therefore we obtain that $\lmasy(D,g) = \log2$ and $\nu_{\max}(D,g) = 2\log2$.

Next, we set 
\[ h(\xi) =
\begin{cases}
-\xi & (0 \leq \xi \leq 1) \\
-1 & (1 \leq \xi),
\end{cases}
\]
for $\xi \in [\eta^\an,x^\an]$ and all closed point $x$ of $X$, which is a continuous function on $X^\an$.
We define a $D$-Green function $g_2$ as $\log\max\{1,|z|\} + h$.
Then we have
\[ \mu_x(g_2-t)=
\begin{cases}
0 & (x \neq x_\infty) \\
1 & (x = x_\infty)
\end{cases}
\]
for $t \leq -1$,
\[ \mu_x(g_2-t)=
\begin{cases}
-1 - t & (x \neq x_\infty) \\
-t & (x = x_\infty)
\end{cases}
\]
for $-1 \leq t \leq 0$ and $\mu_x(g_2-t) = -\infty$ for $t >0$ and all closed point $x$ of $X$.
Hence we obtain that $\nu_{\max}(D,g_2) = -1$ and $g(\eta^\an) = 0$.
\end{rem}

For any integer $n > 0$, let
\[ P_n^{(D,g)}(t) := \frac{\dim_K\Fil^{nt}(H^0(nD),\nm_{ng})}{n^d/d!} \]
where $d = \dim X$.
If there is no confusion, we write it simply $P_n(t)$.
By definition, if $(D,g)$ is $\mu$-finite,
\[ P_n(0) = \frac{\dim_K\Fil^0(H^0(nD),\nm_{ng})}{n^d/d!} = \frac{\dim_KH^0(nD_{\mu(g)})}{n^d/d!}, \]
so we have
\begin{equation}
\label{p1}
\lim_{n \rightarrow +\infty}P_n(0) = \vol(D_{\mu(g)}).
\end{equation}

\begin{lem}
Let $(D,g)$ be an adelic $\R$-Cartier divisor on $X$.
For any $\epsilon \in \R$, we have
\[ P_n^{(D,g-\epsilon)}(t) = P_n^{(D,g)}(t+\epsilon). \]
\end{lem}

\begin{proof}
For any $s \in H^0(nD)$, we have
\[ \|s\|_{n(g -\epsilon)} \leq e^{-nt} \Leftrightarrow \|s\|_{ng} \leq e^{-n(t+\epsilon)}. \]
Hence we get
\[\dim_K\Fil^{nt}(H^0(nD),\nm_{n(g-\epsilon)}) = \dim_K\Fil^{n(t+\epsilon)}(H^0(nD),\nm_{ng}), \]
which implies that $P_n^{(D,g-\epsilon)}(t) = P_n^{(D,g)}(t+\epsilon)$.
\end{proof}

In particular,
\[ P_n^{(D,g)}(t) = P_n^{(D,g-t)}(0). \]
So by the equation (\ref{p1}), we have
\begin{equation}
\label{p2}
\lim_{n \rightarrow +\infty}P_n(t) = \vol(D_{\mu(g-t)})
\end{equation}
for any $t < \nu_{\max}(D,g)$.

If we define
\begin{eqnarray*}
F_{(D,g)}(t) := \left \{
\begin{array}{ll}
\vol(D_{\mu(g-t)}) & (t < \lmasy(D,g)) \\
0 &  (t > \lmasy(D,g)),
\end{array}
\right .
\end{eqnarray*}
we get the following theorem by the equation (\ref{p2}):

\begin{thm}
Let $(D,g)$ be an adelic $\R$-Cartier divisor on $X$.
The sequence $\{ P_n(t) \}_{n \geq 1}$ converges pointwise to $F_{(D,g)}(t)$ on $\R\setminus\{\lmasy(D,g)\}$.
\end{thm}

The sequence $\{ P_n(t) \}_{n \geq 1}$ is uniformly bounded on $(0,\lmasy(D,g))$, so we get the main theorem in this section by using bounded convergence theorem:

\begin{thm}
\label{vol1}
Let $(D,g)$ be an adelic $\R$-Cartier divisor on $X$.
We have
\begin{align*}
\avol(D,g) &= \lim_{n \rightarrow +\infty} \frac{\adeg_+(nD,ng)}{n^{d+1}/(d+1)!} \\
&= (d+1)\int_0^{\lmasy(D,g)} F_{(D,g)}(t) \ \mathrm{d}t.
\end{align*}
\end{thm}

\begin{proof}
By definition,
\[ \adeg_+(nD,ng) = \int_0^{\lm(nD,ng)} \dim_K\Fil^t(H^0(nD),\nm_{ng}) \ \mathrm{d}t. \]
Substituting $t$ for $nt$, we have
\[ \adeg_+(nD,ng) = n\int_0^{\frac{1}{n}\lm(nD,ng)} \dim_K\Fil^{nt}(H^0(nD),\nm_{ng}) \ \mathrm{d}t. \]
Therefore we get
\[ \frac{\adeg_+(nD,ng)}{n^{d+1}/(d+1)!} = (d+1)\int_0^{\lmasy(D,g)} P_n(t) \  \mathrm{d}t. \]
We remark that $\lm(nD,ng)/n \leq \lmasy(D,g)$ and $P_n(t) = 0$ if $t > \lm(nD,ng)/n$.
So by using bounded convergence theorem, we get the conclusion.
\end{proof}

\begin{cor}
\label{vol2}
The arithmetic volume $\avol(\ndot)$ is $(d+1)$-homogeneous.
Namely, for any adelic $\R$-Cartier divisor $(D,g)$ and $a \in \R_{>0}$, we have
\[ \avol(aD,ag) = a^{d+1} \avol(D,g). \]
\end{cor}

\begin{proof}
We have $\lmasy(aD,ag) = a\lmasy(D,g)$ and $F_{(aD,ag)}(at) = a^d F_{(D,g)}(t)$ because the algebraic volume is $d$-homogeneous.
Therefore by Theorem \ref{vol1}, we have
\begin{align*}
\avol(aD,ag) &= (d+1) \int_0^{\lmasy(aD,ag)} F_{(aD,ag)}(t) \ \mathrm{d}t \\
&= a(d+1) \int_0^{\lmasy(D,g)} a^d F_{(D,g)}(t) \ \mathrm{d}t \\
&= a^{d+1} \avol(D,g).
\end{align*}
\end{proof}

\begin{cor}
Let $(D,g)$ be an adelic $\R$-Cartier divisor and $r \in \R_{>0}$.
\begin{enumerate}
\item $\lmasy(D,r^*g) = r\lmasy(D,g)$.
\item $\avol(D,r^*g) = r\avol(D,g)$.
\end{enumerate}
\end{cor}

\begin{proof}
(1) For a non-zero element $f \in H^0(nD) \setminus \{0\}$, we have
\[ |f|_{nr^*g}(x) = \exp(-nrg(x^\frac{1}{r})+\log|f|_x) = \exp(r(-ng(x^\frac{1}{r})+\log|f|_{x^\frac{1}{r}})) = |f|_{ng}(x^\frac{1}{r}). \]
Hence we get $\lm(nD,nr^*g) = r\lm(nD,ng)$ for any positive integer $n$, which implies that $\lmasy(D,r^*g) = r\lmasy(D,g)$.

(2) If $D$ is not big, $\avol(D,g) = 0$ for any $D$-Green function $g$.
So we can assume that $D$ is big.
For $t < \lmasy(D,g)$ and $x \in X^{(1)}$, we have
\begin{align*}
\mu_x(r^*g-t) &= \inf_{\xi \in (\eta^\an,x^\an)} \frac{rg(\xi^\frac{1}{r})-t}{t(\xi)} \\
&= \inf_{\xi \in (\eta^\an,x^\an)} \frac{g(\xi^\frac{1}{r})-t/r}{t(\xi^\frac{1}{r})} = \mu_x(g-t/r),
\end{align*}
which implies that $D_{\mu(r^*g-t)} = D_{\mu(g-t/r)}$ and  hence $F_{(D,r^*g)}(t) = F_{(D,g)}(t/r)$.
Therefore, by Theorem \ref{vol1} and (1), we get
\begin{align*}
\avol(D,r^*g) &= (d+1) \int_0^{r\lmasy(D,g)} F_{(D,g)}(t/r)\ \mathrm{d}t \\
&= (d+1)r \int_0^{\lmasy(D,g)} F_{(D,g)}(t')\ \mathrm{d}t' \quad \left(t' = \frac{t}{r}\right) \\
&= r\avol(D,g).
\end{align*}
\end{proof}

Finally, we prove a simple criterion of the bigness of an adelic $\R$-Cartier divisor.

\begin{thm}[{c.f. \cite[Lemma 1.6]{boucksom2011okounkov} and \cite[Proposition 4.10]{chen2017sufficient}}]
Let $(D,g)$ be an adelic $\R$-Cartier divisor on $X$.
We assume that $D$ is big.
Then the following conditions are equivalent:
\begin{enumerate}
\item $(D,g)$ is big.
\item $\lmasy(D,g) > 0$.
\item For $\forall n \gg 0$, there is a strictly small section of $H^0(nD)$.
\end{enumerate}
\end{thm}

\begin{proof}
(1) $\Rightarrow$ (2)
It follows from Proposition \ref{mbig}.

(2) $\Rightarrow$ (1) 
It is sufficient to show that $D_{\mu(g)}$ is big.
In fact, if $D_{\mu(g)}$ is big, $D_{\mu(g-t)}$ is also big for $t < \lmasy(D,g)$ because $\lmasy(D,g-t) = \lmasy(D,g) - t$ for $t \in \R$.
Then we have
\[ \avol(D,g) = (d+1)\int_0^\infty \vol(D_{\mu(g-t)})\ \mathrm{d}t > 0 \]
by Theorem \ref{vol1}.
Now we prove that $D_{\mu(g)}$ is big.
Since $D$ is big, there is an ample divisor $A$ such that $mD-A$ is effective for some $m \in \Z_{>0}$.
Let $s \in H^0(mD-A) \setminus \{0\}$ be a non-zero section such that the map $H^0(kA) \rightarrow H^0(kmD)$ is given by multiplication by $s^{\otimes k}$ for all $k > 0$.
We denote the image of the map $H^0(kA) \rightarrow H^0(kmD)$ by $V_k$ and $V_0 = K$.
Since the graded ring $\bigoplus_{k=0}^\infty V_k$ is finitely generated, there is $a \in \R$ such that $\|v\|_{kmg} \leq e^{-akm}$ for all $v \in V_k$ and a sufficiently large $k>0$.
Let $\epsilon$ be a real number such that $0 < \epsilon < \lmasy(D,g)$.
Then we can find $p \in \Z_{>0}$ such that there is a non-zero element $s_p \in H^0(pD) \setminus \{0\}$ with $\|v_p\|_{pg} \leq e^{-p\epsilon}$ and $p > -am/\epsilon$ because $\epsilon < \lm(pD,pg)/p$ for a sufficiently large $p>0$.
The image $W_k$ of the composition of the map $H^0(kA) \rightarrow H^0(kmD) \rightarrow H^0(k(m+p)D)$ is given by multiplication by $(ss_p)^{\otimes k}$ for all $k>0$.
Hence for any $w \in W_k$, we can write $w = v \otimes (s_p)^{\otimes k}$ with $v \in V_k$ and we have
\[ \|w\|_{k(m+p)g} \leq \|v\|_{kmg} \ndot \|s_p\|_{pg}^k \leq e^{-akm}e^{-kp\epsilon} = e^{-k(am+p\epsilon)} \leq 1, \]
which implies that $W_k \subset H^0(k(m+p)D_{\mu(g)})$ for a sufficiently large $k > 0$.
Therefore we obtain that $\vol((m+p)D_{\mu(g)}) \geq \vol(A) > 0$, which is required.

(2) $\Rightarrow$ (3)
Since $\lmasy(D,g) > 0$, we have $\lm(nD,ng) > 0$ for a sufficiently large $n > 0$.
Hence there is a non-zero section $s \in H^0(nD) \setminus \{0\}$ such that $\|s\|_{ng} \leq e^{-\lm(nD,ng)} < 1$.

(3) $\Rightarrow$ (2)
Let $s$ be a strictly small section of $H^0(nD)$.
Then we have $\lm(nD,ng) \geq -\log\|s\|_{ng} > 0$.
Therefore we obtain that $\lmasy(D,g) \geq \lm(nD,ng)/n > 0$.
\end{proof}

\subsection{Continuity of $F_{(D,g)}(t)$}

Firstly, we will prove a very useful lemma:

\begin{lem}
\label{con}
Let $V$ be a convex cone and let $f:V \rightarrow \R$ be a concave function.
Namely, for any $v,v' \in V$ and $a,a' \geq 0$, 
\[ f(av+a'v') \geq af(v) + a'f(v'). \]
If $g(t) := v+tv'$ is a map from some open interval $(a,b) \subset \R$ to $V$ for fixed elements $v,v' \in V$, then $f \circ g$ is a concave function on $(a,b)$.
In particular, $f \circ g$ is continuous on $(a,b)$.
\end{lem}

\begin{proof}
For any $t,t' \in (a,b)$ and $0 \leq \epsilon \leq 1$, we have
\begin{align*}
f \circ g(\epsilon t + (1-\epsilon)t') &= f(v+(\epsilon t + (1-\epsilon)t')v') \\
&= f(\epsilon (v+tv') + (1-\epsilon)(v+t'v')) \\
& \geq \epsilon f(v+tv') + (1-\epsilon) f(v+t'v') \\
&= \epsilon f \circ g(t) + (1-\epsilon)f \circ g(t').
\end{align*}
\end{proof}

For an adelic $\R$-Cartier divisor $(D,g)$, it follows immediately from the above lemma that $\mu_x(g-t)$ is a continuous concave function on $(-\infty,\lmasy(D,g))$ for every $x \in X^{(1)}$.
Hence we get the following proposition:

\begin{prop}
\label{D1}
Let $(D,g)$ be an adelic $\R$-Cartier divisor on $X$.
For any $t,t' < \lmasy(D,g)$ and $0 \leq \epsilon \leq 1$, we have
\[ D_{\mu(g-(\epsilon t + (1-\epsilon)t'))} \geq \epsilon D_{\mu(g-t)} + (1-\epsilon)D_{\mu(g-t')}. \]
\end{prop}

\begin{thm}
Let $(D,g)$ be an adelic $\R$-Cartier divisor on $X$ and $d = \dim X$.
Then $F_{(D,g)}(t)$ is a $d$-concave function on $(-\infty,\lmasy(D,g))$, that is, $F_{(D,g)}(t)^{1/d}$ is concave on $(-\infty,\lmasy(D,g))$.
In particular, $F_{(D,g)}(t)$ is continuous on $\R \setminus \{\lmasy(D,g)\}$.
\end{thm}

\begin{proof}
By definition,
\[ F_{(D,g)}(t) = \vol(D_{\mu(g-t)}) \]
for $t < \lmasy(D,g)$.
Since the algebraic volume is $d$-concave on a big cone, for any $t,t' < \lmasy(D,g)$ and $0 \leq \epsilon \leq 1$, we have
\begin{align*}
F_{(D,g)}(\epsilon t + (1-\epsilon)t')^{\frac{1}{d}} &= \vol(D_{\mu(g - (\epsilon t + (1-\epsilon)t'))})^{\frac{1}{d}} \\
&\geq \vol(\epsilon D_{\mu(g-t)} + (1-\epsilon)D_{\mu(g-t')})^{\frac{1}{d}} \quad (\because \text{Proposition  \ref{D1}}) \\
&\geq \epsilon \ \vol(D_{\mu(g-t)})^{\frac{1}{d}} + (1-\epsilon)\vol(D_{\mu(g-t')})^{\frac{1}{d}} \\
&= \epsilon F_{(D,g)}(t)^{\frac{1}{d}} + (1-\epsilon)F_{(D,g)}(t')^{\frac{1}{d}}.
\end{align*}
\end{proof}

\begin{rem}
In general, we cannot extend $F_{(D,g)}$ to a continuous function on the whole $\R$.
For example, let $X = \Ps_K^1 = \mathrm{Proj}K[T_0,T_1]$, $z=T_1/T_0$, $D = \{T_0=0\}$ and $x_\infty = (0:1)$
Let $g = \log\max\{1,|z|\}$.
Then we have
\[ \mu_x(g-t)=
\begin{cases}
0 & (x \neq x_\infty) \\
1 & (x = x_\infty)
\end{cases}
\]
for $t < \lmasy(D,g) = 0$.
Hence we obtain that
\[ F_{(D,g)}(t) = 
\begin{cases}
1 & (t < 0) \\
0 & (t > 0).
\end{cases}
\]
\end{rem}

\subsection{Continuity of the arithmetic volume}

Firstly, we will prove the continuity of $\lmasy(D,g)$ for an adelic $\R$-Cartier divisor $(D,g)$.

\begin{lem}
\label{lam1}
Let $(D,g),(D',g')$ be adelic $\R$-Cartier divisors on $X$.
We have
\[ \lmasy(D+D',g+g') \geq \lmasy(D,g) + \lmasy(D',g'). \]
\end{lem}

\begin{proof}
For any integers $n,n' > 0$, there are non-zero elements $s \in H^0(nD) \setminus \{0\}$ and $s' \in H^0(n'D') \setminus \{0\}$ such that
\[ \|s\|_{ng} \leq e^{-\lm(nD,ng)} \ \text{and} \ \|s'\|_{n'g'} \leq e^{-\lm(n'D',n'g')}. \]
Since $s^{\otimes n'} \otimes s'^{\otimes n} \in H^0(nn'(D+D')) \setminus \{0\}$, we have
\[ \|s^{\otimes n'} \otimes s'^{\otimes n}\|_{nn'(g+g')} \leq (\|s\|_{ng})^{n'}(\|s'\|_{n'g'})^n \leq e^{-n'\lm(nD,ng)-n\lm(n'D',n'g')}, \]
which implies
\[ \frac{1}{nn'}\lm(nn'(D+D'),nn'(g+g')) \geq \frac{1}{n}\lm(nD,ng) + \frac{1}{n'}\lm(n'D',n'g'). \]
Since $\lmasy(D,g) \geq \lm(nD,ng)/n$, we have
\[ \lmasy(D+D',g+g') \geq \frac{1}{n}\lm(nD,ng) + \frac{1}{n'}\lm(n'D',n'g'). \]
Taking a supremum with respect to $n,n'$, we complete the proof.
\end{proof}

\begin{prop}
\label{lam2}
Let $\overline{D}=(D,g),\overline{D}'=(D',g')$ be adelic $\R$-Cartier divisors on $X$.
We assume $D$ is big.
Then $\lambda(t) := \lmasy(\overline{D}+t\overline{D}')$ is a real-valued function on some open interval $(a,b) \subset \R$ containing 0, and concave on $(a,b)$.
In particular $\lambda(t)$ is continuous on $(a,b)$.
\end{prop}

\begin{proof}
Since $D$ is big, $D+tD'$ is big for $|t| \ll 1$, which implies that $\lambda(t)$ is definable on a sufficiently small open neighborhood of 0.
Moreover, using Lemma \ref{lam1}, we can prove the concavity of $\lambda(t)$ by Lemma \ref{con}.
\end{proof}

Next, we prove the continuity of the arithmetic volume $\avol(\ndot)$.
Let $(D,g),(D',g')$ be adelic $\R$-Cartier divisors on $X$ and we assume $D$ is big.
We set
\[ (D_\epsilon,g_\epsilon) := (D,g) + \epsilon(D',g'), \]
and
\begin{eqnarray*}
F_\epsilon(t) := \left \{
\begin{array}{ll}
\vol((D_\epsilon)_{\mu(g_\epsilon-t)}) & (t < \lmasy(D_\epsilon,g_\epsilon)) \\
0 &  (t > \lambda_{{\max}}^{\asy}(D_\epsilon,g_\epsilon)).
\end{array}
\right .
\end{eqnarray*}
We remark that this function is well-defined if $|\epsilon| \ll 1$ by Proposition \ref{lam2}.

\begin{prop}
The function $F_\epsilon(t)$ converges pointwise to $F_{(D,g)}(t)$ on $\R \setminus \{ \lmasy(D,g) \}$ as $|\epsilon| \rightarrow 0$.
More precisely, for any $t \in \R \setminus \{ \lmasy(D,g) \}$, $F_\epsilon(t)$ is continuous with respect to $\epsilon$ on a sufficiently small open neighborhood of $\epsilon = 0$.
\end{prop}

\begin{proof}
We first assume $t > \lmasy(D,g)$.
By Proposition \ref{lam2}, there is $\delta > 0$ such that $\lmasy(D_\epsilon,g_\epsilon) < t$ if $|\epsilon| < \delta$.
Then $F_\epsilon(t) = F_{(D,g)}(t) = 0$, which is required.

Next we assume $t < \lmasy(D,g)$.
Similarly, there is $\delta > 0$ such that $\lmasy(D_\epsilon,g_\epsilon) > t$ if $|\epsilon| < \delta$.
Then $F_\epsilon(t)$ is $d$-concave with respect to $\epsilon$ on $(-\delta,\delta)$, where $d = \dim X$.
In fact, by Lemma \ref{con}, for any $\epsilon,\epsilon' \in (-\delta,\delta)$ and $0 \leq \zeta \leq 1$, we have
\[ (D_{\zeta\epsilon+(1-\zeta)\epsilon'})_{\mu(g_{\zeta\epsilon+(1-\zeta)\epsilon'}-t)} \geq \zeta(D_\epsilon)_{\mu(g_\epsilon-t)} + (1-\zeta)(D_{\epsilon'})_{\mu(g_{\epsilon'}-t)}. \]
Therefore $F_\epsilon(t)$ is $d$-concave with respect to $\epsilon$ on $(-\delta,\delta)$ because $F_\epsilon(t) = \vol((D_\epsilon)_{\mu(g_\epsilon-t)})$ and the algebraic volume is $d$-concave.
In particular, $F_\epsilon(t)$ is continuous with respect to $\epsilon$ on $(-\delta,\delta)$.
\end{proof}

Since $F_\epsilon(t)$ is uniformly bounded with respect to $\epsilon$ and
\[ \avol(D_\epsilon,g_\epsilon) = (d+1) \int_0^{+\infty} F_\epsilon(t) \ \mathrm{d}t \]
by Theorem \ref{vol1}, we get the continuity of the arithmetic volume by bounded convergence theorem:

\begin{thm}
\label{cvol}
Let $\overline{D}=(D,g),\overline{D}'=(D',g')$ be adelic $\R$-Cartier divisors on $X$.
We assume $D$ is big.
Then $\avol(\overline{D}+\epsilon \overline{D}')$ converges to $\avol(\overline{D})$ as $|\epsilon| \rightarrow 0$.
\end{thm}

\subsection{Log concavity of the arithmetic volume}

Firstly, we will prove some inequalities:

\begin{lem}
\label{inq1}
Let $a,b,p$ and $\epsilon$ be real numbers such that $a,b \geq 0$, $p > 0$ and $0 < \epsilon < 1$.
Then we have the following inequality:
\[ (\epsilon a^p + (1-\epsilon)b^p)^{\frac{1}{p}} \geq a^\epsilon b^{1-\epsilon} \geq \min \{a,b\}. \]
\end{lem}

\begin{proof}
If $ab = 0$, the assertion is clear, so we assume that $a,b > 0$. 
Moreover, the inequality $a^\epsilon b^{1-\epsilon} \geq \min \{a,b\}$ is also clear.
Now, we will show the first inequality.
Since $\log x$ is concave on $(0,+\infty)$, we have
\[ \log (\epsilon x + (1-\epsilon)y) \geq \epsilon \log x + (1-\epsilon) \log y \]
for any $x,y > 0$.
Substituting $x$ for $a^p$ and $y$ for $b^p$,

\begin{align*}
\log (\epsilon a^p + (1-\epsilon)b^p) \geq \epsilon \log a^p + (1-\epsilon) \log b^p 
\Longleftrightarrow &\log (\epsilon a^p + (1-\epsilon)b^p)^\frac{1}{p} \geq \log a^\epsilon b^{1-\epsilon} \\
\Longleftrightarrow & (\epsilon a^p + (1-\epsilon)b^p)^{\frac{1}{p}} \geq a^\epsilon b^{1-\epsilon}
\end{align*}
\end{proof}

\begin{lem}
\label{inq2}
Let $V$ be a convex cone. Let $f : V \rightarrow (0,+\infty)$ be a non-negative $d$-homogeneous function for some $d > 0$, that is,
\[ f(av) = a^df(v) \]
for any $a > 0$ and $v \in V$.
Then the following conditions are equivalent:
\begin{enumerate}
\item $f$ is $d$-concave, that is,
\[ f(\epsilon v + (1-\epsilon)v')^\frac{1}{d} \geq \epsilon f(v)^\frac{1}{d} + (1-\epsilon) f(v')^\frac{1}{d} \]
for every $v,v' \in V$ and $0 \leq \epsilon \leq 1$.
\item $f(\epsilon v + (1-\epsilon)v') \geq \min \{ f(v), f(v') \}$ for every $v,v' \in V$ and $0 \leq \epsilon \leq 1$.
\end{enumerate}
\end{lem}

\begin{proof}
Firstly, we assume $(1)$ and we can assume $\min \{f(v),f(v')\} = f(v)$.
Then we have 
\[ f(\epsilon v + (1-\epsilon)v')^\frac{1}{d} \geq \epsilon f(v)^\frac{1}{d} + (1-\epsilon) f(v')^\frac{1}{d} \geq f(v)^\frac{1}{d}. \]
Raising both sides to $d$-th power, we have
\[ f(\epsilon v + (1-\epsilon)v') \geq f(v) = \min \{ f(v), f(v') \}. \]

Next we assume $(2)$.
If we set 
\[ w = f(v)^{-\frac{1}{d}}v,\ w'= f(v')^{-\frac{1}{d}}v',\ \epsilon = \frac{f(v)^\frac{1}{d}}{f(v)^\frac{1}{d} + f(v')^\frac{1}{d}}, \]
we have
\begin{align*}
\epsilon w + (1-\epsilon) w' &= \frac{1}{f(v)^\frac{1}{d} + f(v')^\frac{1}{d}}(v+v'), \\
\min \{f(w),f(w')\} &= 1.
\end{align*}
By the inequality $(2)$ for $w,w'$ and $\epsilon$, we have
\begin{align*}
(f(v)^\frac{1}{d} + f(v')^\frac{1}{d})^{-d}f(v+v') \geq 1 
\Longleftrightarrow & f(v+v') \geq (f(v)^\frac{1}{d} + f(v')^\frac{1}{d})^d \\
\Longleftrightarrow & f(v+v')^\frac{1}{d} \geq f(v)^\frac{1}{d} + f(v')^\frac{1}{d},
\end{align*}
which implies the inequality $(1)$ because $f$ is $d$-homogeneous.
\end{proof}

Moreover, we will use the following inequality so called ``Pr{\'e}kopa-Leindler inequality''.
It was proved by Pr{\'e}kopa \cite{1971logarithmic} \cite{prekopa1973logarithmic} and Leindler \cite{leindler1972certain} (for detail, see \cite{gardner2002brunn}).

\begin{thm}[Pr{\'e}kopa-Leindler inequality]
\label{inq3}
Let $0<\epsilon<1$ and $f,g,h:\R^n \rightarrow [0,+\infty)$ be measurable functions.
We assume
\[ h(\epsilon x + (1-\epsilon)y) \geq f(x)^\epsilon g(y)^{1-\epsilon} \]
for any $x,y \in \R^n$.
Then we have $||h||_1 \geq ||f||_1^\epsilon ||g||_1^{1-\epsilon}$, that is,
\[ \int_{\R^n} h \ \mathrm{d}\nu \geq \left( \int_{\R^n} f \ \mathrm{d}\nu \right)^\epsilon \left( \int_{\R^n} g \ \mathrm{d}\nu \right)^{1-\epsilon} \]
where $\nu$ is the Lubesgue measure on $\R^n$.
\end{thm}

Now, we start to prove the log concavity of $\avol(\ndot)$.

\begin{thm}
\label{cvol2}
The arithmetic volume $\avol(\ndot)$ is $(d+1)$-concave for $d = \dim X$.
More precisely, for any big adelic $\R$-Cartier divisors $(D,g),(D',g')$, we have
\[ \avol(D+D',g+g')^\frac{1}{d+1} \geq \avol(D,g)^\frac{1}{d+1} + \avol(D',g')^\frac{1}{d+1}. \]
\end{thm}

\begin{proof}
For $0 < \epsilon < 1$, we set
\[ (D_\epsilon,g_\epsilon) := \epsilon(D,g) + (1-\epsilon)(D',g'), \]
and
\begin{eqnarray*}
\Theta_{(D,g)}(t) := \left \{
\begin{array}{ll}
(d+1)\vol(D_{\mu(g-t)}) & (0 \leq t < \lmasy(D,g)) \\
0 &  (\text{otherwise}).
\end{array}
\right .
\end{eqnarray*}
Then, we have
\begin{equation}
\label{ve1}
\avol(D,g) = ||\Theta_{(D,g)}||_1,\ \avol(D',g') = ||\Theta_{(D',g')}||_1,\ \avol(D_\epsilon,g_\epsilon) = ||\Theta_{(D_\epsilon,g_\epsilon)}||_1
\end{equation}
by Theorem \ref{vol1}.
We claim that
\begin{equation}
\label{ve2}
\Theta_{(D_\epsilon,g_\epsilon)}(\epsilon x + (1-\epsilon)y) \geq \Theta_{(D,g)}(x)^\epsilon \Theta_{(D',g')}(y)^{1-\epsilon} \quad \text{for any} \ x,y \in \R.  
\end{equation}
In fact, if $x<0, \lmasy(D,g) \leq x, y<0$ or $\lmasy(D',g') \leq y$, we have $\Theta_{(D,g)}(x)=0$ or $\Theta_{(D',g')}(y)=0$, so the inequality (\ref{ve2}) is clear in this case.
And if $0 \leq x < \lmasy(D,g)$ and $0 \leq y < \lmasy(D',g')$, we have
\[ \mu_z(g_\epsilon -(\epsilon x + (1-\epsilon)y)) \geq \epsilon \mu_z(g-x) + (1-\epsilon)\mu_z(g'-y) \]
for any $z \in X^{(1)}$, which implies that
\[ (D_\epsilon)_{\mu(g_\epsilon - (\epsilon x + (1-\epsilon)y))} \geq \epsilon D_{\mu(g-x)} + (1-\epsilon)D'_{\mu(g'-y)}. \]
Since the algebraic volume is $d$-concave, we obtain
\[ \vol((D_\epsilon)_{\mu(g_\epsilon - (\epsilon x + (1-\epsilon)y))})^\frac{1}{d} \geq \epsilon \ \vol(D_{\mu(g-x)})^\frac{1}{d} + (1-\epsilon)\vol(D'_{\mu(g'-y)})^\frac{1}{d}. \]
By Lemma \ref{inq1}, we get
\[ \vol((D_\epsilon)_{\mu(g_\epsilon - (\epsilon x + (1-\epsilon)y))}) \geq \vol(D_{\mu(g-x)})^\epsilon \vol(D'_{\mu(g'-y)})^{1-\epsilon}, \]
which is equivalent to the inequality (\ref{ve2}).
Therefore by Pr{\'e}kopa-Leindler inequality, we have $||\Theta_{(D_\epsilon,g_\epsilon)}||_1 \geq ||\Theta_{(D,g)}||_1^\epsilon ||\Theta_{(D',g')}||_1^{1-\epsilon}$.
By Lemma \ref{inq1} again, we have $||\Theta_{(D_\epsilon,g_\epsilon)}||_1 \geq \min \{ ||\Theta_{(D,g)}||_1, ||\Theta_{(D',g')}||_1 \}$, which is the inequality
\[ \avol(D_\epsilon,g_\epsilon) \geq \min \{ \avol(D,g), \avol(D',g') \} \]
by (\ref{ve1}).
Since the arithmetic volume is $(d+1)$-homogeneous by Corollary \ref{vol2}, we have
\[ \avol(D_\epsilon,g_\epsilon)^\frac{1}{d+1} \geq \epsilon \ \avol(D,g)^\frac{1}{d+1} + (1-\epsilon) \avol(D',g')^\frac{1}{d+1}, \]
by Lemma \ref{inq2}, which completes the proof.
\end{proof}

\end{document}